\documentclass[11pt]{article}
\usepackage[a4paper,margin=27mm]{geometry}
\usepackage{amsmath,amssymb,amsthm,mathtools,bm}
\usepackage[T1]{fontenc}
\usepackage{lmodern,microtype,booktabs,array}
\usepackage{xcolor,graphicx}
\usepackage[colorlinks=true,linkcolor=blue!40!black,citecolor=blue!40!black,urlcolor=blue!40!black]{hyperref}
\numberwithin{equation}{section}
\newtheorem{theorem}{Theorem}[section]
\newtheorem{proposition}[theorem]{Proposition}
\newtheorem{lemma}[theorem]{Lemma}
\newtheorem{corollary}[theorem]{Corollary}
\newtheorem{assumption}[theorem]{Assumption}
\theoremstyle{definition}
\newtheorem{definition}[theorem]{Definition}
\theoremstyle{remark}
\newtheorem{remark}[theorem]{Remark}
\newcommand{\E}{\mathbb E}
\newcommand{\Prob}{\mathbb P}
\newcommand{\R}{\mathbb R}
\newcommand{\C}{\mathbb C}
\newcommand{\Tr}{\operatorname{Tr}}

\newcommand{\diag}{\operatorname{diag}}
\newcommand{\HS}{\mathrm{HS}}
\newcommand{\op}{\mathrm{op}}

\newcommand{\TW}{\mathrm{TW}_1}
\newcommand{\dd}{\mathrm d}
\newcommand{\ii}{\mathrm i}
\newcommand{\cK}{\mathcal K}

\newcommand{\norm}[1]{\left\lVert#1\right\rVert}

\DeclareMathOperator{\supp}{supp}
\allowdisplaybreaks[2]
\title{Tracy--Widom law for sample covariance matrices under quadratic form concentration}
\author{Guangchen Chen\footnotemark[1] \qquad Xing Lyu\footnotemark[2]\\
\\
The University of Hong Kong}
\date{}
\begin{document}
\maketitle
\footnotetext[1]{Supported by Hong Kong RGC GRF 17304225, u3014790@connect.hku.hk}
\footnotetext[2]{Supported by Hong Kong RGC GRF 17304225, u35xing@connect.hku.hk}
\begin{abstract}
We study sample covariance matrices formed from vectors $q_i = \Sigma^{1/2}u_i \in \mathbb{R}^M$, $1 \leq i \leq N$, where $u_1, \ldots, u_N$ are independent, centered, isotropic random vectors whose coordinates may not be independent. We assume that the random vectors $q_i$ satisfy a concentration bound for their quadratic forms $q_i^T A q_i$. Under this assumption and a regular right edge condition on $\Sigma$, we establish edge universality and obtain the Tracy--Widom law for the largest eigenvalue of the sample covariance matrix after appropriate centering and scaling. In particular, for log-concave ensembles, our result removes the unconditionality assumption imposed in existing edge universality results~\cite{BX}, resolving the question that originally motivated our work. The concentration assumption also allows us to extend the result beyond the log-concave class.
\end{abstract}

% BEGIN REPLACEMENT SECTIONS
\section{Introduction}

\subsection{Background and matrix model.}

Let $S=(q_1,\ldots,q_N)\in\R^{M\times N}$ be a data matrix whose columns
are independent observations, not necessarily identically distributed. We assume that each column is centered with common covariance $\Sigma$. Thus
\begin{equation}\label{eq:isotropy}
    \E q_i=0, \qquad \E q_iq_i^T=\Sigma. 
\end{equation}
We are interested in the normalized covariance matrix formed by the data matrix
\begin{equation}\label{eq:model}
 H \equiv H_N:=\frac1N\sum_{i=1}^Nq_iq_i^T=XX^T,
 \qquad X=(x_1,\ldots,x_N):=N^{-1/2}S.
\end{equation}
Throughout the paper, we consider the proportional regime $M/N\equiv y\to y_\infty\in(0,\infty)\setminus\{1\}$ as $N$ tends to infinity.

The main motivation for \eqref{eq:model} arises from estimating the population covariance matrix by the sample covariance matrix in high--dimensional statistics and big data applications. In the low dimensional case, under some moment conditions, the sample covariance matrix $H$ is a consistent estimator of the population covariance matrix $\Sigma$, i.e., $\norm{H - \Sigma}_{\op} \rightarrow 0$ as $N \rightarrow \infty$ provided that $M$ is fixed. However, this is no longer the case in the high dimensional regime when $M$ is proportional to $N$. There is an extensive literature discussing this effect. In the case when each column has identity population covariance and independent entries, the Marchenko--Pastur (MP) law~\cite{MP} shows that the empirical spectral distribution converges weakly to a deterministic limit supported on \([\lambda_-^\infty,\lambda_+^\infty]\) (and possibly $ \{0\}) $, where \(\lambda_\pm^\infty=(1\pm\sqrt{y_\infty})^2\). 
The right edge \(\lambda_+^\infty\) is the threshold for the largest eigenvalue $\lambda_1(H)$, and \(\lambda_{1}(H)\to(1+\sqrt{y_\infty})^2\) almost surely under a finite fourth moment condition~\cite{YinBai}. 
Moreover, our primary focus is on the fluctuations of \(\lambda_{1}(H)\) around the edge, which have been shown to asymptotically follow the Tracy--Widom law with proper centering and scaling, for independent Gaussian entries~\cite{Johansson00, Johnstone01} and for independent entries with general distributions~\cite{PY}. In fact, based on the classical three step strategy~\cite{ErdosYau17}, the largest eigenvalue of many matrix models is proved to exhibit the Tracy--Widom fluctuations. For instance, we refer the reader to~\cite{ErdosSchleinYau2012, LeeYin2014, LandonYau2017, LeeSchnelli16, KnowlesYin2017, AltErdosKrugerSchroder2020, PillaiYin2012, BaoPanZhou2012, BaoPanZhou2015, DingYang2018, Peche09, Soshnikov02}.

For a population covariance $\Sigma$, the MP law is replaced by the deformed MP law. Its regular right edge and the square root decay rate depend on $\Sigma$ and $M/N$, see~\cite{KnowlesYin2017}.

Our first motivation was to remove the additional assumption in proving edge universality for a log-concave ensemble in~\cite{BX}. Inspired by~\cite{BaiZhou2008, PY, FanMaPaquetteWang26, MaMisiakiewicz2026}, we further consider the random vectors exhibiting a concentration of quadratic form 
\begin{equation}\label{eq:ref1}
    |q_i^T A q_i - \Tr(\Sigma A)| \prec \norm{A}_{\HS},
\end{equation}
uniformly over deterministic matrices $A \in \C^{M \times M}$. This condition \eqref{eq:ref1} is also achieved by Gaussian vectors and log-concave vectors. We refer to~\cite{FanMaPaquetteWang26, MaMisiakiewicz2026} for more examples. Our theorem uses the moment form of the concentration given in \eqref{eq:quadratic-form-moments} for convenience.

\subsection{Main results.}
We denote the largest eigenvalue of $H$ by $\lambda_1(H)$. Our main result shows that the fluctuation of the largest eigenvalue around a regular right edge is universal across columns with common covariance $\Sigma$ satisfying the following assumption. The Gaussian reference matrix has the same population covariance $\Sigma$. Its largest eigenvalue is known to follow the Tracy--Widom distribution after proper centering and scaling~\cite{KnowlesYin2017}.

Throughout the paper, our basic assumption on the random vectors is as follows.
\begin{assumption}\label{assu.1.1}
We make the following three assumptions on the vectors $q_i$.
\begin{enumerate}
\item[(i)] \textit{(On $q$)}
We assume that the random vectors $q_1,\ldots,q_N$ are independent, but not necessarily identically distributed, and satisfy
\begin{align*}
\mathbb{E}q_i=0,
\qquad
\mathbb{E}q_iq_i^T=\Sigma,
\qquad 1\le i\le N.
\end{align*}
\item[(ii)] \textit{(On $y$)}
We assume that $M/N\equiv y\longrightarrow y^\infty\in(0,\infty)\setminus\{1\}$, as $N\to\infty$.
\item[(iii)] \textit{(Concentration of quadratic forms)}
For every fixed $p\ge2$, there exists a constant $C_p>0$, uniform in $i$, $M$ and $N$, such that, for every deterministic symmetric matrix
$A\in\mathbb{C}^{M\times M}$,
\begin{equation}\label{eq:quadratic-form-moments}
    \bigl\|q_i^T A q_i-\Tr(\Sigma A) \bigr\|_{p}
    \le C_p\|A\|_{\HS},
    \qquad 1\le i\le N,
\end{equation}
where $\norm A_\HS=(\Tr A^*A)^{1/2}$.
\end{enumerate}
\end{assumption}

In particular, we do not assume that the fourth or higher moments agree between different columns. Throughout the paper, our estimates hold uniformly in the column index,
with constants depending on the common concentration constants $C_p$.

We next state the assumption on the population covariance. Let
$\sigma_1\ge\cdots\ge\sigma_M$ be the eigenvalues of $\Sigma$ and set
\begin{equation}\label{eq:population-f}
 f(m):=-\frac1m+\frac1N\sum_{\alpha=1}^M
 \frac{\sigma_\alpha}{1+\sigma_\alpha m}.
\end{equation}
There is a unique $m_+\in(-\sigma_1^{-1},0)$ satisfying $f'(m_+)=0$.
Define the right edge and its scaling parameter by
\begin{equation}\label{eq:population-edge}
 \lambda_+:=f(m_+),\qquad
 \gamma_+:=\left(\frac{2}{f''(m_+)}\right)^{1/3}.
\end{equation}

\begin{assumption}\label{assu.population}
\textit{(On $\Sigma$)}
The matrix $\Sigma=\Sigma_N$ is deterministic, real symmetric and positive definite.
For some constant $\tau>0$, independent of $N$, we assume
\begin{equation}\label{eq:population-regularity}
 \tau I_M\le\Sigma\le\tau^{-1}I_M,
 \qquad 1+\sigma_1m_+\ge\tau.
\end{equation}
\end{assumption}

The second condition in \eqref{eq:population-regularity} was first introduced in~\cite{ElKaroui2007}, and is the regularity condition at the rightmost edge in the anisotropic local law in~\cite{KnowlesYin2017}. Together with the bounds on $\Sigma$ and $y$, this condition ensures
a uniformly regular right edge separated from the other edges. 

\begin{theorem}[Universality of largest eigenvalue]\label{thm.universality}\label{universality thm}
Suppose that Assumptions~\ref{assu.1.1} and~\ref{assu.population} hold for the independent, not necessarily identically distributed vectors $q_1,\ldots,q_N$. Let
$X^w$ have independent $N(0,N^{-1})$ entries and set
$W:=\Sigma^{1/2}X^w(X^w)^T\Sigma^{1/2}$. Then, for every $x\in\R$,
\begin{align*}
 \lim_{N\to\infty}\Bigl\{
 \Prob^H\bigl(\gamma_+N^{2/3}(\lambda_1(H)-\lambda_+)\le x\bigr)
 -\Prob^W\bigl(\gamma_+N^{2/3}(\lambda_1(W)-\lambda_+)\le x\bigr)
 \Bigr\}=0.
\end{align*}
\end{theorem}

It is known that the largest eigenvalue of the Gaussian sample covariance matrix with a regular right edge has the Tracy--Widom limit~\cite[Corollary~3.19(i)]{KnowlesYin2017}. Combining this result with Theorem~\ref{universality thm}, we have the following corollary.
\begin{corollary}\label{cor.universality}
Under the same conditions as Theorem~\ref{universality thm}, we have
\begin{align*}
 \gamma_+N^{2/3}(\lambda_1(H)-\lambda_+)\ \longrightarrow\ \mathrm{TW}_1
\end{align*}
in distribution, where $\mathrm{TW}_1$ denotes the Tracy--Widom distribution.
\end{corollary}

\begin{remark}
When $\Sigma=I_M$, we have
\[
 m_+=-\frac1{1+\sqrt y},\qquad
 \lambda_+=(1+\sqrt y)^2,\qquad
 \gamma_+=\frac{y^{1/6}}{(1+\sqrt y)^{4/3}}.
\]
The normalization in Corollary~\ref{cor.universality} is therefore
\[
 \frac{N\lambda_1(H)-(\sqrt M+\sqrt N)^2}
 {(\sqrt M+\sqrt N)(1/\sqrt M+1/\sqrt N)^{1/3}}.
\]
The formulation above does not require the empirical eigenvalue distribution of $\Sigma_N$ to converge.
\end{remark}

\subsection{Proof strategy.}
Traditional proofs of edge universality follow the three step strategy developed by the Erd\H{o}s--Yau school, including local law, dynamical approach, and Green function comparison. We refer to the monograph~\cite{ErdosYau17} for a standard structure of the proof. The recent results of~\cite{FanMaPaquetteWang26,MaMisiakiewicz2026} provide the local law and spectral rigidity for the present population model under quadratic form concentration.

Following the approach of~\cite{PY,BX}, we consider a small Gaussian perturbation. Let $X^w$ be an $M\times N$ matrix with independent $N(0,N^{-1})$ entries, independent of $S=(q_1,\ldots,q_N)$. For $0\le t\le\tau/2$, set
\begin{align}\label{eq.Guassian interpolation}
 X(t):=N^{-1/2}(I_M-t\Sigma^{-1})^{1/2}S+\sqrt t\,X^w,
 \qquad H(t):=X(t)X(t)^T.
\end{align}
In particular, $X(0)=X$ and each column of $X(t)$ has covariance $\Sigma/N$.
Let $t_0=N^{-1/3+\varepsilon}$ for a sufficiently small constant $\varepsilon>0$, and write $\widetilde X=X(t_0)$ and $\widetilde H=H(t_0)$.
Conditionally on $S$, this is an additive Gaussian perturbation of
$N^{-1/2}(I_M-t_0\Sigma^{-1})^{1/2}S$. The regularity check and the edge analysis in~\cite{DY} therefore apply. We verify that the random edge and scaling obtained from this perturbation agree with $\lambda_+$ and $\gamma_+$ at the required precision.

To complete the proof, it remains to conduct Green function comparison between $H$ and $\widetilde H$. We replace one column of $X$ by the corresponding column of $\widetilde X$ at each step. 

Using the convention in \cite{PY}, and with the crude bound from the assumption, we can express the comparison result by the difference of two products of quadratic forms
\begin{equation}\label{1}
    \E\left[\left(x^T \mathcal{G}(z_1) x\right)\cdot \left(x^T \mathcal{G}(z_2) x\right)\right] 
-
\E\left[\left(\tilde{x}^T \mathcal{G}(z_1) \tilde{x}\right)\cdot \left(\tilde{x}^T \mathcal{G}(z_2) \tilde{x}\right)\right] 
\end{equation}
where $\mathcal{G}(z_1)$ and $\mathcal{G}(z_2)$ are two resolvents of the interpolated data matrix with two spectral parameters $z_1$ and $z_2$.
The major obstacle is that we cannot get a desired bound for this term using merely the typical size of the resolvents obtained from the local law. In order to get a more precise estimate of this term, the resolvent matrix in the quadratic form shall be further expanded. This is the main technical task of our paper. After careful error analysis, we finally construct a self consistent equation for \eqref{1}. By solving this equation, we succeed in controlling \eqref{1} to our desired scale and complete the proof of edge universality. 

\subsection{Notations and conventions.}
Regarding different norms, for a vector $u$, $|u|$ denotes the Euclidean norm. For a matrix $A$, $\norm A_\op$ is its operator norm and $\norm A_\HS=(\Tr A^*A)^{1/2}$ its Hilbert-Schmidt norm. For a random scalar $Z$, we write $\norm Z_p=(\E|Z|^p)^{1/p}$. 

For two positive quantities $\mathcal{X}$ and $\mathcal{Y}$,
we write $\mathcal{X}\lesssim\mathcal{Y}$ if there exists a
constant $C>0$
such that $\mathcal{X}\le C\mathcal{Y}$.
Throughout the paper, we also use the notation of stochastic domination introduced in the monograph~\cite{ErdosYau17}.

\noindent
\begin{definition}[Stochastic Domination]
\label{stochastic domination}

    Let
\[
\mathsf{X}
=
\left(
\mathsf{X}^{(N)}(u):
N\in\mathbb{N},\ u\in\mathsf{U}^{(N)}
\right),
\qquad
\mathsf{Y}
=
\left(
\mathsf{Y}^{(N)}(u):
N\in\mathbb{N},\ u\in\mathsf{U}^{(N)}
\right)
\]
be two families of random variables, where $\mathsf{Y}$ is
nonnegative, and $\mathsf{U}^{(N)}$ is a possibly
$N$-dependent parameter set.

We say that $\mathsf{X}$ is stochastically dominated by
$\mathsf{Y}$, uniformly in $u$, if for all small $\epsilon>0$
and large $D>0$,
\[
\sup_{u\in\mathsf{U}^{(N)}}
\mathbb{P}\left(
\left|\mathsf{X}^{(N)}(u)\right|
>
N^\epsilon\mathsf{Y}^{(N)}(u)
\right)
\le N^{-D}
\]
for large enough $N>N_0(\epsilon,D)$.
If $\mathsf{X}$ is stochastically dominated by $\mathsf{Y}$,
uniformly in $u$, we use the notation
$\mathsf{X}\prec\mathsf{Y}$, or equivalently
$\mathsf{X}=O_{\prec}(\mathsf{Y})$.
Note that in the special case when $\mathsf{X}$ and
$\mathsf{Y}$ are deterministic, $\mathsf{X}\prec\mathsf{Y}$
means that for any given $\epsilon>0$,
\[
\left|\mathsf{X}^{(N)}(u)\right|
\le N^\epsilon\mathsf{Y}^{(N)}(u)
\]
uniformly in $u$, for all sufficiently large
$N\ge N_0(\epsilon)$.
\end{definition}

\subsection{Organization.}
The paper is organized as follows. Section~\ref{sec.preliminary} is devoted to the preliminaries, which are crucial for our later discussion.
In Section~\ref{sec.main-proof}, we prove
Theorem~\ref{thm.universality} based on Green function comparison.
Section~\ref{sec.weighted-cumulant} establishes the estimates
needed for Green function comparison by deriving and solving a self-consistent equation.
Finally, the proofs of several auxiliary estimates are relegated to the appendices.

\section{Preliminaries}
\label{sec.preliminary}

\subsection{Settings.}
Besides $H=XX^T$, we use the companion matrix $\mathcal H=X^TX$.
For $z=E+\ii\eta\in\C\setminus\R$, define
\begin{align*}
 \mathcal{G}_1(z)&=(H-zI_M)^{-1},&
 \mathcal{G}_2(z)&=(\mathcal H-zI_N)^{-1},\\
 m_{1N}(z)&=M^{-1}\Tr\mathcal{G}_1(z),&
 m_{2N}(z)&=N^{-1}\Tr\mathcal{G}_2(z).
\end{align*}
The two matrices have the same nonzero eigenvalues. Their empirical
eigenvalue laws differ only in the number of zero eigenvalues and in
the normalization of the trace. In particular,
\[
 m_{2N}(z)=ym_{1N}(z)+\frac{y-1}{z}.
\]
The deterministic law corresponding to $H$ is denoted as $\nu_{y,1}$,
and for $\mathcal H$ it is $\nu_{y,2}$. These are the deformed
Marchenko--Pastur laws associated with $\Sigma$ and $y$; we suppress
the dependence on $\Sigma$ in the notation. Their Stieltjes transforms
are denoted by $m_1(z)$ and $m_2(z)$. With the function $f$ defined in
Section~1, $m_2$ is the unique solution in $\C^+$ of $z=f(m_2(z))$.
The deterministic matrix corresponding to $\mathcal G_1$ is
\[
 \Pi(z):=-\bigl[z(I_M+m_2(z)\Sigma)\bigr]^{-1}.
\]
Thus
\begin{equation}\label{eq:m_1}
 \begin{split}
 z&=-\frac1{m_2(z)}
       +\frac1N\sum_{\alpha=1}^M
          \frac{\sigma_\alpha}{1+\sigma_\alpha m_2(z)},\\
 m_1(z)&=\frac1M\Tr\Pi(z),\qquad
 m_2(z)=ym_1(z)+\frac{y-1}{z}.
 \end{split}
\end{equation}
Here $\sigma_1,\ldots,\sigma_M$ are the eigenvalues of $\Sigma$.
The transforms in the lower half-plane are defined by complex
conjugation. We will also use the deterministic quantities
\begin{gather}
 a_0(z):=\frac1N\Tr\bigl(\Sigma\Pi(z)\bigr),\qquad
 d_0(z):=1+a_0(z)=-\frac1{zm_2(z)},
 \label{eq:population-denominator}\\
 zI_M-\frac{\Sigma}{d_0(z)}=-\Pi(z)^{-1}.
 \label{eq:population-stability}
\end{gather}
The last two identities follow directly from \eqref{eq:m_1}.

Assumption~\ref{assu.population} implies that $\lambda_+$ is a regular
right edge in the sense of \cite[Definition~2.7]{KnowlesYin2017}.
In particular, for some fixed $\delta>0$, the interval
$[\lambda_+-2\delta,\lambda_+]$ belongs to the rightmost component
of the support, $\lambda_+-2\delta>0$, and
\[
 \frac{\dd\nu_{y,2}}{\dd E}(E)\asymp\sqrt{\lambda_+-E},
 \qquad \lambda_+-2\delta\le E<\lambda_+.
\]
The same estimate holds for $\nu_{y,1}$ since their densities away
from zero differ by the factor $y$. We choose $\delta$ small enough
throughout the paper. For $|E-\lambda_+|\le\delta$ and $0<\eta\le1$,
the deterministic estimates from \cite[Appendix~A]{KnowlesYin2017}
give
\begin{align}\label{eq:regular-edge-bounds}
 |m_2(z)|\asymp1,\qquad |d_0(z)|\asymp1,\qquad
 \|\Pi(z)\|_\op+\|\Pi(z)^{-1}\|_\op\lesssim1.
\end{align}
Moreover, with $\kappa:=|E-\lambda_+|$,
\begin{align}\label{eq:regular-edge-density}
 \Im m_2(z)\asymp
 \begin{cases}
  \sqrt{\kappa+\eta},&E\le\lambda_+,\\
  \eta/\sqrt{\kappa+\eta},&E\ge\lambda_+.
 \end{cases}
\end{align}

For a set $\mathbb T\subset\{1,\ldots,N\}$, let
$X^{(\mathbb T)}$ be the matrix obtained by deleting the columns
indexed by $\mathbb T$. The quantities
$\mathcal G_1^{(\mathbb T)}(z)$,
$\mathcal G_2^{(\mathbb T)}(z)$, etc. are defined similarly using
$X^{(\mathbb T)}$. We also set
\[
 m_1^{(\mathbb T)}(z):=\frac1M\Tr\mathcal G_1^{(\mathbb T)}(z).
\]

\subsection{Local law and rigidity.}
The following estimates are obtained from the results of
\cite{FanMaPaquetteWang26,MaMisiakiewicz2026}. Under Assumption~\ref{assu.1.1} and Assumption~\ref{assu.population}, our model satisfies Assumptions~1 and~2 of both papers \cite{FanMaPaquetteWang26,MaMisiakiewicz2026}, including the concentration condition for quadratic forms. 
Moreover, the spectral domain near the right edge is regular by \cite[Lemma~2.4(b)]{FanMaPaquetteWang26}. Therefore, the local law and rigidity
estimates stated below are consequences of
\cite[Theorem~2.5 and Remark~2.6]{FanMaPaquetteWang26} and
\cite[Corollary~2.7]{FanMaPaquetteWang26}, respectively.
Set the domain
\begin{align}\label{eq.domain}
 D(\epsilon):=\{z=E+\ii\eta\in\C:
 |E-\lambda_+|\le\delta,\quad
 N^{-1+\epsilon}\le\eta\le1\},
\end{align}
and define
\[
 \Psi(z):=\sqrt{\frac{\Im m_2(z)}{N\eta}}+\frac1{N\eta}.
\]

\begin{theorem}[Local MP law]
\label{local MP law}
Under Assumption~\ref{assu.1.1} and Assumption~\ref{assu.population},
uniformly on $D(\epsilon)$, we have 
\begin{enumerate}
\item[(i)]
The Stieltjes transforms of the empirical spectral distributions
satisfy
\begin{align}\label{trace gap}
 |m_{1N}(z)-m_1(z)|+|m_{2N}(z)-m_2(z)|
 \prec\frac{\Psi(z)^2}{\sqrt{\kappa+\eta}+\Psi(z)}
 \lesssim\frac1{N\eta}.
\end{align}
\item[(ii)]
The individual matrix elements of the Green function
satisfy
\begin{align}\label{entry wise local law}
 \bigl|[\mathcal G_2(z)]_{ij}-m_2(z)\delta_{ij}\bigr|
 \prec\Psi(z).
\end{align}
\item[(iii)]
For any deterministic unit vectors $u,v\in\C^M$, we have
\begin{align}\label{anisotropic local law}
 \bigl|u^*\bigl(\mathcal G_1(z)-\Pi(z)\bigr)v\bigr|
 \prec\Psi(z).
\end{align}
In particular,
\begin{align}\label{weighted trace local law}
 \left|\frac1N\Tr\bigl(\Sigma\mathcal G_1(z)\bigr)-a_0(z)\right|
 \prec\Psi(z).
\end{align}
\end{enumerate}
\end{theorem}

\begin{proof}
The first two statements are
\cite[Theorem~2.5 and Remark~2.6]{FanMaPaquetteWang26}.
The anisotropic estimate is
\cite[Theorem~2.4]{MaMisiakiewicz2026}.
To obtain \eqref{weighted trace local law}, apply
\eqref{anisotropic local law} to an orthonormal eigenbasis of
$\Sigma$, sum the resulting diagonal estimates, and use
$M\asymp N$ and $\|\Sigma\|_\op\lesssim1$.
\end{proof}

Let the classical locations of the eigenvalues $\gamma_j$
be defined as follows:
\begin{align}\label{def of gammaj}
 \int_{\gamma_j}^{\lambda_+}\nu_{y,1}(\dd x)
 =\frac{j-1/2}{M},\qquad 1\le j\le M\wedge N.
\end{align}

\begin{theorem}[Spectral rigidity]\label{rigidity}
Under Assumption~\ref{assu.1.1} and Assumption~\ref{assu.population},
there is a fixed constant $\delta_1>0$ such that, uniformly for
$1\le j\le\delta_1N$,
\begin{align}\label{eigenvalue rigidity}
 |\lambda_j(H)-\gamma_j|\prec N^{-2/3}j^{-1/3}.
\end{align}
Moreover, we have
\begin{align}\label{eq:right-edge-confinement}
 |\lambda_1(H)-\lambda_+|\prec N^{-2/3},
 \qquad \|H\|_\op\prec1.
\end{align}
For every fixed $b>0$, with overwhelming probability all eigenvalues
of $H$ belong to the $b$-neighborhood of $\supp\nu_{y,1}$.
\end{theorem}

\begin{proof}
These statements follow from
\cite[Corollary~2.7(a)--(c)]{FanMaPaquetteWang26} and the
square root behavior of the density near $\lambda_+$.
Choose $\delta_1$ small enough that the classical locations of the
first $\delta_1N$ eigenvalues belong to
$[\lambda_+-\delta,\lambda_+]$. In particular,
$\lambda_+-\gamma_1\asymp N^{-2/3}$, so the first estimate in
\eqref{eq:right-edge-confinement} also follows from
\eqref{eigenvalue rigidity}. The deterministic support is uniformly
bounded, which gives the operator norm bound.
\end{proof}

\subsection{Interpolation matrices and basic estimates}
\label{sec:basic-scales}
Recall the definition of the Gaussian interpolation matrix $\tilde{X}$.
For $\gamma=0,\ldots,N$, let $X_\gamma$ be the random matrix
whose first $\gamma$ columns are those of $\widetilde X$ and whose remaining columns are those of $X$.
Thus $X_0=X$ and $X_N=\widetilde X$.
Set
\[
H_\gamma:=X_\gamma X_\gamma^*,\qquad
\mathcal H_\gamma:=X_\gamma^*X_\gamma,
\]
and define
\[
G_{1,\gamma}(z):=(H_\gamma-zI_M)^{-1},\qquad
G_{2,\gamma}(z):=(\mathcal H_\gamma-zI_N)^{-1},\qquad
m_{2N,\gamma}(z):=\frac1N\operatorname{Tr}G_{2,\gamma}(z).
\]
We also consider the minors $X_\gamma^{(\mathbb{T})}$ and
$H_\gamma^{(\mathbb{T})}$, and they are defined in the same way as
$X^{(\mathbb{T})}$ and $H^{(\mathbb{T})}$, respectively.

\begin{lemma}[Preservation by convolution]
\label{lem:interpolation-stability}
Suppose that $q_1,\ldots,q_N$ satisfy Assumption~\ref{assu.1.1}.
For each $t \in [0,t_*]$, where $t_*:=\tau/2$, define
\[
q_\gamma(t)
:=(I_M-t\Sigma^{-1})^{1/2}q_\gamma
+\sqrt{t}\,g_\gamma,
\qquad
g_\gamma:=\sqrt N\,x_\gamma^w,
\qquad 1\le\gamma\le N.
\]
Then $q_1(t),\ldots,q_N(t)$ also satisfy
Assumption~\ref{assu.1.1}.
\end{lemma}

\begin{proof}[Proof of Lemma \ref{lem:interpolation-stability}]
Set
\[
B_t:=(I_M-t\Sigma^{-1})^{1/2},\qquad
C_t:=\sqrt t\,I_M.
\]
These are deterministic real symmetric matrices satisfying
$B_t\Sigma B_t^T+C_tC_t^T=\Sigma$ and
$\|B_t\|_{\op}+\|C_t\|_{\op}\le C$. Thus, we know that the vectors $q_i(t)$ are independent and centered, with common covariance $\Sigma$.
It remains to verify \eqref{eq:quadratic-form-moments}.

Fix $p\ge2$ and let $A$ be a deterministic real symmetric matrix.
Ignoring the column index $\gamma$, write $q(t)=B_tq+C_tg$.
Then
\begin{align*}
q(t)^TAq(t)-\operatorname{Tr}(\Sigma A)
={}&q^TB_t^TAB_tq-\operatorname{Tr}(\Sigma B_t^TAB_t)
\\ &+g^TC_t^TAC_tg-\operatorname{Tr}(C_t^TAC_t)
+2q^TB_t^TAC_tg.
\end{align*}
The first two terms on the right side are controlled by Assumption~\ref{assu.1.1}
and the Gaussian moment formula, since
$\|B_t^TAB_t\|_{\HS}+\|C_t^TAC_t\|_{\HS}\le C\|A\|_{\HS}$.
To estimate the mixed term, note that, for every deterministic
$u\in\mathbb R^M$, Assumption~\ref{assu.1.1} applied to $uu^T$ yields
\[
\|u^Tq\|_{2p}^2
=\|(u^T q)^2\|_{p}
=\|q^Tuu^Tq\|_p
\le(C_p+\|\Sigma\|_{\op})\norm{u}_2^2.
\]
where the last inequality follows from 
$\Tr(uu^T)=\|u\|_2^2$. Then, for every deterministic real matrix $D$, we write $d_j^T$ for the rows of $D$,
\[
\bigl\|\|Dq\|_2\bigr\|_{p}^2
=
\left\|\sum_j|d_j^T q|^2\right\|_{p/2}
\le
\sum_j\|d_j^T q\|_{p}^2
\le
C_p\sum_j\|d_j\|_2^2
=
C_p\|D\|_{\mathrm{HS}}^2.
\]
Conditioning on $q$ and using the Gaussian moment formula, we obtain
\[
\|q^TB_t^TAC_tg\|_p
\le C_p\bigl\||C_t^TAB_tq|\bigr\|_p
\le C_p\|C_t^TAB_t\|_{\HS}
\le C_p\|A\|_{\HS}.
\]
Combining these estimates completes the proof, while the bound for complex symmetric $A$ follows by applying the real bound to its real and imaginary parts.
\end{proof}

We next state two lemmas that provide the basic estimates used below.
Their proofs are postponed to Appendix \ref{app:basic scales} and
\ref{app:denominator}, respectively.

\begin{lemma}[Basic scales]\label{lem:res}
Fix a positive integer $K_0$.
For $\gamma\in\{0,\ldots,N\}$ and
$\mathbb T\subset\{1,\ldots,N\}$ with $|\mathbb T|\le K_0$,
write $R:=G_{1,\gamma}^{(\mathbb T)}(z)$.
Uniformly in $\gamma$, $\mathbb T$, and $z$ satisfying
\begin{equation}\label{eq:domain1}
|E-\lambda_+|\le N^{-2/3+\varepsilon},\qquad
N^{-2/3-\varepsilon}\le|\eta|\le N^{-2/3+\varepsilon},
\end{equation}
we have
\begin{equation}\label{eq:schatten}
\operatorname{Tr}|R|\prec N,\quad
\operatorname{Tr}|R|^{3/2}\prec N^{1+C\varepsilon},\quad
\operatorname{Tr}|R|^s\prec N^{2s/3+C_s\varepsilon}\quad(s>3/2),
\end{equation}
where $|R|=(R^*R)^{1/2}$ and $s$ is fixed.
In particular, taking $s=2,4$ yields
\begin{equation}\label{eq.HS-bounds}
\|R\|_{\mathrm{HS}}\prec N^{2/3+C\varepsilon},\qquad
\|R^2\|_{\mathrm{HS}}\prec N^{4/3+C\varepsilon}.
\end{equation}
Moreover, let $x=N^{-1/2}q$, where $q$ satisfies
Assumption~\ref{assu.1.1} and is independent of
$X_\gamma^{(\mathbb T)}$.
For every fixed $p\ge1$, we have
\begin{equation}\label{eq:primitive}
\begin{aligned}
\lVert x^TRx-N^{-1}\operatorname{Tr}(\Sigma R)\rVert_p
&\le C_pN^{-1/3+C\varepsilon},\\
\lVert Rx\rVert_p&\le C_pN^{1/6+C\varepsilon},\qquad
\lVert x\rVert_p\le C_p,\\
\lVert x^TR^2x\rVert_p&\le C_pN^{1/3+C\varepsilon},\qquad
\lVert x^TR^3x\rVert_p\le C_pN^{1+C\varepsilon}.
\end{aligned}
\end{equation}
These bounds are uniform in the same parameters.
The corresponding Schatten and trace bounds remain valid when a fixed
number of deterministic matrices with bounded operator norms are
inserted between the resolvent factors.
\end{lemma}

\begin{lemma}\label{lem:denominator}
Under the assumptions and notation of Lemma~\ref{lem:res}, set
\[
d_0(z):=1+\frac1N\operatorname{Tr}(\Sigma\Pi(z))
=-\frac1{z m_2(z)}.
\]
The following estimates hold uniformly in $\gamma$, $\mathbb T$, and
$z$ satisfying \eqref{eq:domain1}:
\begin{equation}\label{eq:trace-mp}
\begin{gathered}
\left|N^{-1}\operatorname{Tr}(\Sigma R)-(d_0(z)-1)\right|
\prec N^{-1/3+C\varepsilon},\\
|d_0(z)|\asymp1,\qquad
\left\|\left(zI_M-\frac{\Sigma}{d_0(z)}\right)^{-1}\right\|_{\op}
=\|\Pi(z)\|_{\op}\le C.
\end{gathered}
\end{equation}
Moreover, for every fixed $p\ge1$,
\begin{equation}\label{eq:denominator-moments}
\begin{aligned}
\left\|(1+N^{-1}\operatorname{Tr}(\Sigma R))^{-1}\right\|_p
+\left\|(1+x^TRx)^{-1}\right\|_p
&\le C_p,\\
\left\|\frac1{1+x^TRx}-\frac1{d_0(z)}\right\|_p
&\le C_pN^{-1/3+C\varepsilon}.
\end{aligned}
\end{equation}
\end{lemma}

\section{Proof of Theorem \ref{thm.universality}}
\label{sec.main-proof}
Recall the Gaussian interpolation $X(t)$ and
$\widetilde X=X(t_0)$ introduced in \eqref{eq.Guassian interpolation}. Following \cite{BX,PY}, the proof of Theorem \ref{thm.universality} consists of two key steps: establishing the Tracy Widom limit for $H(t_0)$ and comparing its largest eigenvalue distribution
with that of $H=H(0)$. In the sequel, we show the detailed discussion.

\subsection{Edge Universality of \texorpdfstring{$X(t_0)X(t_0)^T$}{the Gaussian-divisible ensemble}}
We begin with establishing the Tracy--widom limit for $H(t_0)$. Fix $t=t_0=N^{-1/3+\varepsilon}$ with $0<\varepsilon<1/12$, and set
\[
 D=\Sigma-tI_M,\qquad
 Y_0=N^{-1/2}(I_M-t\Sigma^{-1})^{1/2}S,\qquad
 V=Y_0Y_0^T.
\]
The lower bound on $\Sigma$ ensures that $D$ is positive definite for
all sufficiently large $N$. Conditionally on $S$, the matrix
$X(t)=Y_0+\sqrt t\,X^w$ is the rectangular matrix Dyson Brownian motion
at time $t$, with initial matrix $Y_0$. Its unconditional column covariance is
$N^{-1}\Sigma$. We use this interpolation in order to keep the
deterministic limiting density fixed while applying the dynamical
result of~\cite{DY}.

The regularity of the right edge is stable under replacing $\Sigma$
by $D$. Denote the right edge for $D$ by $a$, and write
\[
 g(w)=m_{1,D}(w),\qquad h(w)=M^{-1}\Tr(V-wI_M)^{-1}.
\]
The local law and rigidity from Section~\ref{sec:basic-scales} apply
uniformly to $V$. Taking $\eta_*=N^{-2/3+\varepsilon}$, these estimates
and the square root behavior of the deterministic density show that,
with high probability, $V$ is $\eta_*$-regular in the sense of
\cite[Definition~V.1]{DY}. Indeed, rigidity gives
$\lambda_1(V)-a\prec N^{-2/3}=o(\eta_*)$, and the averaged local law, including the exterior bound in \eqref{trace gap}, gives the
required upper and lower estimates for the imaginary part of $h$
on the domains in that definition. Spectral confinement supplies its
remaining conditions. For any fixed $0<\varepsilon_1<\varepsilon$,
\[
 N^{\varepsilon_1}\eta_*\le t^2\le N^{-\varepsilon_1}
\]
for all sufficiently large $N$. Hence
\cite[Theorem~V.3]{DY}, applied conditionally on $S$, gives
\begin{equation}\label{eq.population-conditional-TW}
 \widehat\gamma_tN^{2/3}
 \bigl(\lambda_1(H(t))-\widehat\lambda_{+,t}\bigr)
 \ \Longrightarrow\ \TW.
\end{equation}
The theorem applies to a deterministic initial matrix satisfying the stated regularity bounds. Conditionally on $S$, no distributional assumption on its columns is needed. On the regularity event the constants can be chosen uniformly, so we may average the conditional conclusion over $S$.

Here $\widehat\lambda_{+,t}$ and $\widehat\gamma_t$ are the right edge
and its scaling parameter for the rectangular free convolution
started from the empirical spectral distribution of $V$.
When $M>N$, we apply the same theorem to $y^{-1/2}X(t)^T$ and then
rescale its nonzero eigenvalues by $y$; this gives exactly the
$N$-normalization in \eqref{eq.population-conditional-TW}.

We next compare these random parameters with $\lambda_+$ and
$\gamma_+$. For a Stieltjes transform $u$ define
\[
 \Phi_{t,u}(w)
 =w\bigl(1-tyu(w)\bigr)^2
  +t(1-y)\bigl(1-tyu(w)\bigr).
\]
The subordination equation for the rectangular free convolution is
\begin{equation}\label{eq.population-rfc}
 z=\Phi_{t,u}(w),\qquad
 m_{1,t}(z)=\frac{u(w)}{1-tyu(w)}.
\end{equation}
For $u=g$, this convolution is precisely the deformed MP law for
$\Sigma$. To verify this identity directly, let
$m=m_{2,\Sigma}(z)$, $b=1+ty m_{1,\Sigma}(z)$, and set
\[
 w=zb(1+tm),\qquad k=\frac{m}{1+tm}.
\]
Substitution in the deformed MP equation gives
$w=f_D(k)$ and
$g(w)=m_{1,\Sigma}(z)/b$. Moreover,
$w=zb^2-t(1-y)b$. These identities give
\eqref{eq.population-rfc}, with $m_{1,t}=m_{1,\Sigma}$.
The branch is fixed by the Stieltjes transform normalization at
infinity and analytic continuation.

Let $\xi$ be the rightmost critical point of $\Phi_{t,g}$, and let
$\widehat\xi$ be its counterpart for $\Phi_{t,h}$. The regular square
root expansion of $g$ at $a$ implies
\begin{equation}\label{eq.population-critical-scales}
 \xi-a\asymp t^2,\qquad
 g'(\xi)\asymp t^{-1},\qquad
 \Phi_{t,g}''(\xi)\asymp t^{-2}.
\end{equation}
In particular, $\Phi_{t,g}(\xi)=\lambda_+$, whereas
$\Phi_{t,h}(\widehat\xi)=\widehat\lambda_{+,t}$.

For completeness, we record the estimates needed to compare these
critical points. Uniformly in a sufficiently small disc of radius
$ct^2$ centered at $\xi$, the local law and rigidity imply
\begin{equation}\label{eq.population-exterior-derivatives}
 \bigl|h^{(j)}(w)-g^{(j)}(w)\bigr|
 \prec\frac{1}{Nt^{2j+2}},\qquad j=0,1,2.
\end{equation}
Here is a direct justification that only uses rigidity near the
right edge. Split the spectral integrals at a fixed distance below
$a$ and compare each empirical eigenvalue near $a$ with its
classical location. For $j=1,2$, the quantile estimate gives
\[
 \bigl|h^{(j)}(w)-g^{(j)}(w)\bigr|
 \prec \frac1N\int_0^c\frac{\dd s}{(s+t^2)^{j+2}}+1
 \lesssim\frac{1}{Nt^{2j+2}}+1.
\]
The last term bounds the part a fixed distance below the edge and
is absorbed since $\varepsilon<1/12$. For $j=0$, first apply the
averaged local law at imaginary part $t^2$, which gives
$O_\prec((Nt^2)^{-1})$, and then integrate the $j=1$ estimate along
a vertical segment of length $O(t^2)$. The additional $O(t^2)$
term is absorbed because $Nt^4\le1$. This proves
\eqref{eq.population-exterior-derivatives}.

It follows from \eqref{eq.population-exterior-derivatives} that
\[
 \Phi_{t,h}-\Phi_{t,g}\prec\frac1{Nt},\qquad
 \Phi_{t,h}'-\Phi_{t,g}'\prec\frac1{Nt^3},\qquad
 \Phi_{t,h}''-\Phi_{t,g}''\prec\frac1{Nt^5}
\]
on this disc. Since $Nt^3=N^{3\varepsilon}\to\infty$,
\eqref{eq.population-critical-scales} and the critical point
equations imply
\[
 |\widehat\xi-\xi|\prec\frac1{Nt},\qquad
 |\widehat\lambda_{+,t}-\lambda_+|\prec\frac1{Nt}.
\]
The critical point is the rightmost one by spectral confinement
and the positivity of $\Phi_{t,h}'$ to its right, as in the
support characterization of the rectangular free convolution.
Consequently,
\begin{equation}\label{eq.population-edge-matching}
 N^{2/3}|\widehat\lambda_{+,t}-\lambda_+|
 \prec N^{-\varepsilon}=o(1).
\end{equation}

The scaling parameter can be checked from the same equations.
The coefficient of $\sqrt{\lambda_+-E}/\pi$ in the density of
the $M$-normalized law in \eqref{eq.population-rfc} is
\[
 a_t=
 \frac{g'(\xi)}{(1-tyg(\xi))^2}
 \sqrt{\frac{2}{\Phi_{t,g}''(\xi)}}.
\]
Thus $\gamma_+=(ya_t)^{2/3}$. The random coefficient is given by
the same formula with $(g,\xi)$ replaced by
$(h,\widehat\xi)$. Using
\eqref{eq.population-exterior-derivatives},
$g''=O(t^{-3})$, and $\Phi_{t,g}'''=O(t^{-4})$, we obtain
\begin{equation}\label{eq.population-scale-matching}
 \left|\frac{\widehat\gamma_t}{\gamma_+}-1\right|
 \prec\frac1{Nt^3}=N^{-3\varepsilon}=o(1).
\end{equation}
Combining \eqref{eq.population-conditional-TW},
\eqref{eq.population-edge-matching}, and
\eqref{eq.population-scale-matching} proves
\[
 \gamma_+N^{2/3}\bigl(\lambda_1(H(t_0))-\lambda_+\bigr)
 \ \Longrightarrow\ \TW.
\]
This completes the first step.

\subsection{Green function comparison between \texorpdfstring{$\tilde{X}$}{X̃} and\texorpdfstring{$X$}{X}}

As shown in \cite{BX,PY}, to prove Theorem \ref{thm.universality}, it suffices to establish a Green function comparisons for $\tilde{H}$ and $H$ on the edge, i.e., Proposition \ref{Green function comparison thm}.  Let $\tilde{\mathcal{G}}_{2}(z) : =  (\tilde{\mathcal{H}} - zI)^{-1}$ be the Green function of $\tilde{\mathcal{H}}$ with $\tilde{\mathcal{H}} := \tilde{X}^*\tilde{X}$, and the corresponding Stieltjes transform is denoted by $\tilde{m}_{2N}(z) := N^{-1} \Tr \tilde{\mathcal{G}}_{2}(z)$. 

\begin{proposition}[Green function comparison] \label{Green function comparison thm}
	Let $F: \mathbb{R} \to \mathbb{R}$ be a function whose derivatives satisfy
	\begin{align*}
		\max_x |F^{(\alpha)}(x)|(|x|+1)^{-C_1} \le C_1, \quad  \alpha = 1,2,3,4
	\end{align*}
	for some constant $C_1 > 0$. Then there exist $\epsilon_0 > 0$, $N_0 \in \mathbb{N}$ and $\delta > 0$ depending on $C_1$ and the constants in the assumptions such that for any $\epsilon < \epsilon_0$, $N \ge N_0$ and real numbers $E, E_1$ and $E_2$ satisfying 
	\begin{align*}
		|E - \lambda_+| \le N^{-2/3+\epsilon}, \quad |E_1 - \lambda_+| \le N^{-2/3+\epsilon}, \quad |E_2 - \lambda_+| \le N^{-2/3+\epsilon},
	\end{align*}
	and $\eta_0 = N^{-2/3-\epsilon}$, we have
	\begin{align}\label{comparison 1}
		|\E F(N\eta_0\Im \tilde{m}_{2N}(z)) - \E F(N\eta_0\Im m_{2N}(z)) | \le C N^{-1/6+C\epsilon}, \quad z = E + \mathrm{i}\eta_0,
	\end{align}
	and
	\begin{align}\label{comparison 2}
		\left| \E F\left( N\int_{E_1}^{E_2}\mathrm{d}y \Im \tilde{m}_{2N}(y+\mathrm{i}\eta_0) \right)- \E F\left( N\int_{E_1}^{E_2}\mathrm{d}y \Im m_{2N}(y+\mathrm{i}\eta_0) \right) \right|\le C N^{-1/6+C\epsilon}
	\end{align}
	for some constant $C > 0$.
\end{proposition}

\begin{proof}[Proof of Proposition \ref{Green function comparison thm}]
 In the sequel, we show the proof of (\ref{comparison 1}) only. The proof of (\ref{comparison 2}) is similar, and thus we omit it. Recall the definition of matrices $X_\gamma$ introduced in \eqref{eq.Guassian interpolation}. To prove \eqref{comparison 1}, we follow the strategy in \cite{BX,PY}: we replace the columns of $X$ by those of $\tilde{X}$ one by one and then estimate the one-step difference between the Green function
functionals associated with $X_\gamma$ and $X_{\gamma-1}$.

Since $X_0=X$ and $X_N=\widetilde X$, we have the
telescoping sum
\begin{align}\label{telescoping sum}
	\E F(N\eta_0\Im \tilde{m}_{2N}(z)) - \E F(N\eta_0\Im m_{2N}(z)) = \sum_{\gamma = 1}^N \E F(N\eta_0\Im m_{2N,\gamma}(z)) - \E F(N\eta_0\Im m_{2N,\gamma-1}(z)).
\end{align}

We can then complete the proof of Proposition \ref{Green function comparison thm} by establishing the following lemma. 
\begin{lemma}\label{telescoping error lemma}
Under the same condition as Proposition \ref{Green function comparison thm}, for any $\gamma \in \{1,\cdots,N \}$, we have 
\begin{align}\label{telescoping error}
	|\E F(\eta_0\Im \Tr \mathcal{G}_{2,\gamma}(z)) - \E F(\eta_0\Im \Tr \mathcal{G}_{2,\gamma-1}(z))| \lesssim N^{-7/6+C\epsilon}.
\end{align}
\end{lemma}
The proof of Lemma \ref{telescoping error lemma} is postponed to Section~\ref{sec.telescoping error lemma}. By Lemma \ref{telescoping error lemma}, the telescoping sum (\ref{telescoping sum}) can be bounded as follows
\begin{align*}
	|\E F(N\eta_0\Im \tilde{m}_{2N}(z)) - \E F(N\eta_0\Im m_{2N}(z))| \lesssim N^{-1/6 + C\epsilon},
\end{align*}
which implies (\ref{comparison 1}). Similarly, one can show (\ref{comparison 2}), which completes the proof of Proposition \ref{Green function comparison thm}.
\end{proof}

To prepare for the proof of Lemma \ref{telescoping error lemma}, we need to derive a Taylor expansion for the difference in \eqref{telescoping error}. Following the approach in \cite{BX,PY}, we choose to expand them 
around the common point $\mu_\gamma$, where $\mu_\gamma$ is defined by
\begin{align*}
 \mu_\gamma
 := \eta_0\Im\Tr\mathcal G_{1,\gamma}^{(\gamma)}(z) + \eta_0(M-N)\Im z^{-1}.
\end{align*}
Note that $X_\gamma$ and $X_{\gamma-1}$ share the same minor
$X_\gamma^{(\gamma)}(=X_{\gamma-1}^{(\gamma)})$. Using the Sherman Morrison formula and taking traces, we have
\begin{align*}
 \Tr\mathcal G_{2,\gamma}(z)
 &= \Tr\mathcal G_{1,\gamma}^{(\gamma)}(z)+\frac{M-N}{z}
    -\frac{
      \bigl\langle \tilde x_\gamma,
      \bigl(\mathcal G_{1,\gamma}^{(\gamma)}(z)\bigr)^2
      \tilde x_\gamma\bigr\rangle
    }{
      1+\bigl\langle \tilde x_\gamma,
      \mathcal G_{1,\gamma}^{(\gamma)}(z)
      \tilde x_\gamma\bigr\rangle
    },\\
 \Tr\mathcal G_{2,\gamma-1}(z)
 &= \Tr\mathcal G_{1,\gamma}^{(\gamma)}(z)+\frac{M-N}{z}
    -\frac{
      \bigl\langle x_\gamma,
      \bigl(\mathcal G_{1,\gamma}^{(\gamma)}(z)\bigr)^2
      x_\gamma\bigr\rangle
    }{
      1+\bigl\langle x_\gamma,
      \mathcal G_{1,\gamma}^{(\gamma)}(z)
      x_\gamma\bigr\rangle
    }.
\end{align*}
Therefore, for $v_\gamma\in\{x_\gamma,\tilde x_\gamma\}$, by defining
\begin{align*}
 d_\gamma
 &:=1+N^{-1}\Tr\Sigma\mathcal G_{1,\gamma}^{(\gamma)}(z),
 \qquad
 \xi(v_\gamma):=\bigl\langle v_\gamma,
 \mathcal G_{1,\gamma}^{(\gamma)}(z)v_\gamma\bigr\rangle
 -N^{-1}\Tr\Sigma\mathcal G_{1,\gamma}^{(\gamma)}(z),\\
 a(v_\gamma)
 &:=\bigl\langle v_\gamma,
 \bigl(\mathcal G_{1,\gamma}^{(\gamma)}(z)\bigr)^2v_\gamma
 \bigr\rangle,
 \qquad
 \Delta(v_\gamma)
 :=-\eta_0\Im\frac{a(v_\gamma)}{d_\gamma+\xi(v_\gamma)}
 =-\eta_0\Im
 \frac{
 \bigl\langle v_\gamma,
 \bigl(\mathcal G_{1,\gamma}^{(\gamma)}(z)\bigr)^2v_\gamma
 \bigr\rangle
 }{
 1+\bigl\langle v_\gamma,
 \mathcal G_{1,\gamma}^{(\gamma)}(z)v_\gamma
 \bigr\rangle
 }.
\end{align*}
we can rewrite the difference in \eqref{telescoping error} as
\begin{align}
 &\E F\bigl(\eta_0\Im\Tr\mathcal G_{2,\gamma}(z)\bigr)
 -\E F\bigl(\eta_0\Im\Tr\mathcal G_{2,\gamma-1}(z)\bigr)
 \notag\\
 &\quad=
 \E F\bigl(\mu_\gamma+\Delta(\tilde x_\gamma)\bigr)
 -\E F\bigl(\mu_\gamma+\Delta(x_\gamma)\bigr)
 \notag\\
 &\quad=
 \E\bigl[
 F\bigl(\mu_\gamma+\Delta(\tilde x_\gamma)\bigr)
 -F(\mu_\gamma)
 \bigr]
 -\E\bigl[
 F\bigl(\mu_\gamma+\Delta(x_\gamma)\bigr)
 -F(\mu_\gamma)
 \bigr].
 \label{eq.taylordifference}
\end{align}

To determine the contribution of $\Delta(v_\gamma)$ to the Taylor expansion of $F$, we next expand it in powers of $\xi(v_\gamma)$. Using the identity
\begin{align*}
 \frac{1}{d+\xi}
 =d^{-1}-d^{-2}\xi+d^{-3}\xi^2
  -d^{-3}(d+\xi)^{-1}\xi^3,
\end{align*}
and defining the real random variables
\begin{align*}
 J_k(v_\gamma)
 :=\Im\left[d^{-k-1}a(v_\gamma)\xi(v_\gamma)^k\right],
 \qquad k=0,1,2,
\end{align*}
With the bounds on Lemma~\ref{lem:res}, we get
\begin{align}\label{eq. delta decomposition}
 \Delta(v_\gamma)
 =-\eta_0J_0(v_\gamma)
   +\eta_0J_1(v_\gamma)
   -\eta_0J_2(v_\gamma)
   +O_{\prec}(N^{-4/3+C\epsilon}),
\end{align}
where
\begin{align*}
 \eta_0J_k(v_\gamma) \prec N^{-(k+1)/3+C\epsilon},
 \qquad
 \Delta(v_\gamma)
\prec N^{-1/3+C\epsilon}.
\end{align*}
Hence, substituting \eqref{eq. delta decomposition} into the third-order Taylor expansion of $F$ at $\mu_\gamma$ and using $J_k$'s bound, we obtain
\begin{align}
 &F\bigl(\mu_\gamma+\Delta(v_\gamma)\bigr)-F(\mu_\gamma)
 \notag\\
 &\quad=
 -\eta_0F'(\mu_\gamma)J_0(v_\gamma)
 +\eta_0F'(\mu_\gamma)J_1(v_\gamma)
 -\eta_0F'(\mu_\gamma)J_2(v_\gamma)
 \notag\\
 &\qquad
 +\frac{\eta_0^2}{2}F''(\mu_\gamma)J_0(v_\gamma)^2
 -\eta_0^2F''(\mu_\gamma)J_0(v_\gamma)J_1(v_\gamma)
 \notag\\
 &\qquad
 -\frac{\eta_0^3}{6}F'''(\mu_\gamma)J_0(v_\gamma)^3
 +O_\prec\bigl(N^{-4/3+C\epsilon}\bigr).
 \label{eq.functional-expansion}
\end{align}
Consequently, by replacing $v_\gamma$ with $\tilde x_\gamma$ and $x_\gamma$ respectively in \eqref{eq.functional-expansion}, and taking the results back into \eqref{eq.taylordifference}, we get
\begin{align}
 \E F\bigl(\eta_0\Im\Tr\mathcal G_{2,\gamma}(z)\bigr)
 -\E F\bigl(\eta_0\Im\Tr\mathcal G_{2,\gamma-1}(z)\bigr)=
 \sum_{k=1}^{6}\mathcal T_{k,\gamma}
 +O\bigl(N^{-4/3+C\epsilon}\bigr),
 \label{eq.six-term-comparison}
\end{align}
where
\begin{equation}\label{eq.tauk defi}
\begin{aligned}
 \mathcal T_{1,\gamma}
 &=-\eta_0\E\left[
 F'(\mu_\gamma)
 \bigl(J_0(\tilde x_\gamma)-J_0(x_\gamma)\bigr)
 \right],\\
 \mathcal T_{2,\gamma}
 &=\eta_0\E\left[
 F'(\mu_\gamma)
 \bigl(J_1(\tilde x_\gamma)-J_1(x_\gamma)\bigr)
 \right],\\
 \mathcal T_{3,\gamma}
 &=-\eta_0\E\left[
 F'(\mu_\gamma)
 \bigl(J_2(\tilde x_\gamma)-J_2(x_\gamma)\bigr)
 \right],\\
 \mathcal T_{4,\gamma}
 &=\frac{\eta_0^2}{2}\E\left[
 F''(\mu_\gamma)
 \bigl(J_0(\tilde x_\gamma)^2-J_0(x_\gamma)^2\bigr)
 \right],\\
 \mathcal T_{5,\gamma}
 &=-\eta_0^2\E\left[
 F''(\mu_\gamma)
 \bigl(J_0(\tilde x_\gamma)J_1(\tilde x_\gamma)
       -J_0(x_\gamma)J_1(x_\gamma)\bigr)
 \right],\\
 \mathcal T_{6,\gamma}
 &=-\frac{\eta_0^3}{6}\E\left[
 F'''(\mu_\gamma)
 \bigl(J_0(\tilde x_\gamma)^3-J_0(x_\gamma)^3\bigr)
 \right].
\end{aligned}
\end{equation}
By \eqref{eq.six-term-comparison}, the proof of Lemma~\ref{telescoping error lemma} reduces to estimating the six terms $\mathcal T_{k,\gamma}$, $k=1,\ldots,6$, uniformly in $\gamma$. In the next section, we develop the tools required to estimate the terms $\mathcal T_{k,\gamma}$ .

\subsection{Analysis of the Difference Terms}
\label{sec.comparison-term-analysis}
In this section, we reformulate the difference terms $\mathcal T_{k,\gamma}$ defined in \eqref{eq.tauk defi}. Rather than estimating them directly, we decompose them into two manageable forms. We begin by analyzing the difference structures of $J_k(\tilde x_\gamma)-J_k(x_\gamma)$. 

Using $\Im z=(z-\overline z)/(2\mathrm i)$, we set
$z_+=z$ and $z_-=\overline z$. For $\sigma\in\{+1,-1\}$ and
$v_\gamma\in\{x_\gamma,\tilde x_\gamma\}$, we define
\begin{equation*}
\begin{gathered}
 d_{\gamma,\sigma}
 =1+N^{-1}\Tr\Sigma\mathcal G_{1,\gamma}^{(\gamma)}(z_\sigma),
 \quad a_\sigma(v_\gamma)
 =\bigl\langle v_\gamma,
 \bigl(\mathcal G_{1,\gamma}^{(\gamma)}(z_\sigma)\bigr)^2
 v_\gamma\bigr\rangle,
 \\
 \xi_\sigma(v_\gamma)
 =\bigl\langle v_\gamma,
 \mathcal G_{1,\gamma}^{(\gamma)}(z_\sigma)v_\gamma
 \bigr\rangle
 -N^{-1}\Tr\Sigma\mathcal G_{1,\gamma}^{(\gamma)}(z_\sigma),
 \quad
 h_{k,\gamma,\sigma}
 =\frac{\sigma}{2\mathrm i}\,d_{\gamma,\sigma}^{-k-1}.
\end{gathered}
\end{equation*}
While
$d_{\gamma,-}=\overline{d_{\gamma,+}}$,
$a_-(v_\gamma)=\overline{a_+(v_\gamma)}$ and
$\xi_-(v_\gamma)=\overline{\xi_+(v_\gamma)}$,
the coefficients $h_{k,\gamma,\sigma}$ allow us to rewrite $J_k$ as
\begin{align*}
 J_k(v_\gamma)
 =\sum_{\sigma=\pm1}
 h_{k,\gamma,\sigma}
 a_\sigma(v_\gamma)\xi_\sigma(v_\gamma)^k,
 \qquad k=0,1,2.
 %\label{eq:Jsign}
\end{align*}

Let $\E_\gamma$ denote expectation on $(x_\gamma,x_\gamma^w)$. Since $h_{k,\gamma,\sigma}$ is fixed under $\E_\gamma$, the conditional expectation of $J_k(\tilde x_\gamma)-J_k(x_\gamma)$ reduce to those involving $a_\sigma(\tilde x)\xi_\sigma(\tilde x)^k - a_\sigma(x_\gamma)\xi_\sigma(x_\gamma)^k$. To analyze these differences, we now state the moment identities and comparison estimates at the second, fourth and sixth order given in Lemma~\ref{lem.moment-identities} and \ref{lem.centered-triple}. For a deterministic symmetric matrix
$A\in\mathbb{C}^{M\times M}$, define
\begin{align*}
 Q_A(q):=q^TAq-\Tr(\Sigma A),
 \qquad Q_A^0(g):=g^TAg-\Tr A.
\end{align*}
Recall that $q_\gamma=\sqrt N x_\gamma$ and
$q_\gamma(t)=\sqrt N x_\gamma(t)=S_tq_\gamma+\sqrt t\,g_\gamma$,
where $S_t=(I_M-t\Sigma^{-1})^{1/2}$. In this section we take
$0\le t\le t_*:=\tau/2$, with $\tau$ from Assumption~\ref{assu.population}.
By covariance matching and Assumption~\ref{assu.1.1},
we have
\begin{align*}
\E_\gamma Q_A(q_\gamma)
=\E_\gamma Q_A(q_\gamma(t))=0,
\qquad
\bigl\|Q_A(q_\gamma)\bigr\|_p
+\bigl\|Q_A(q_\gamma(t))\bigr\|_p
\le C_p\|A\|_{\HS}.
\end{align*}
For each column $\gamma$ and symmetric matrices
$A,B\in\mathbb C^{M\times M}$ held fixed under $\E_\gamma$, define the fourth-cumulant quantity of $q_\gamma$
\begin{align*}
K_\gamma^0[A,B]
:=&\operatorname{Cov}_\gamma
\bigl(q_\gamma^TAq_\gamma,q_\gamma^TBq_\gamma\bigr)
-\operatorname{Cov}_\gamma
\bigl(
g_\gamma^T\Sigma^{1/2}A\Sigma^{1/2}g_\gamma,\,
g_\gamma^T\Sigma^{1/2}B\Sigma^{1/2}g_\gamma
\bigr)\\
=&\E_\gamma\bigl[Q_A(q_\gamma)Q_B(q_\gamma)\bigr]
-2\Tr(\Sigma A\Sigma B).
\end{align*}
which measures the deviation from the Gaussian covariance of the quadratic forms associated with $A$ and $B$. Assumption~\ref{assu.1.1} and the boundedness of $\Sigma$ imply $|K_\gamma^0[A,B]|\le C\|A\|_{\HS}\|B\|_{\HS}$.
To track the change in fourth-cumulants quantity along the interpolation, we further define another fourth-cumulant difference needed below, for $t > 0$
\begin{align*}
 K_{\gamma,t}[A,B]:=\sum_{a,b,c,d}(K_{\gamma,t})_{abcd}A_{ab}B_{cd},
 \quad
 (K_{\gamma,t})_{abcd}
 :=\frac1t\left\{
 \kappa_4(q_\gamma)_{abcd}
 -\kappa_4(q_\gamma(t))_{abcd}\right\},
\end{align*}
At $t=0$, we set $K_{\gamma,0} =0$. Here $\kappa_4(v)$ is denoted as
\begin{align*}
 \kappa_4(v)_{abcd}
 :=\E[v_av_bv_cv_d]
 -\Sigma_{ab}\Sigma_{cd}
 -\Sigma_{ac}\Sigma_{bd}
 -\Sigma_{ad}\Sigma_{bc}
\end{align*}
for a centered vector $v$ with covariance $\Sigma$. Indeed, set $\mathcal B_t(A)=S_tAS_t$, we have
\begin{align*}
 K_{\gamma,t}[A,B]
 =\frac{K_\gamma^0[\mathcal B_0(A),\mathcal B_0(B)]
       -K_\gamma^0[\mathcal B_t(A),\mathcal B_t(B)]}{t}.
\end{align*}
In the rest of the proof we write $K_\gamma:=K_{\gamma,t_0}$. These fourth-cumulant quantities play a crucial role in our proof, and we will discuss it in
Section~\ref{sec.weighted-cumulant}.

\begin{lemma}
\label{lem.moment-identities}
Let $A,B$ be complex symmetric matrices held fixed under
$\E_\gamma$, we have
\begin{align}
 \E_\gamma\left[
 \langle x_\gamma,Ax_\gamma\rangle
 -\langle\tilde x_\gamma,A\tilde x_\gamma\rangle
 \right]=0,
 \label{eq.second-moment}
\end{align}
and, for every $t\in[0,t_*]$,
\begin{align}
 \E_\gamma\left[
 Q_A(q_\gamma)Q_B(q_\gamma)
 -Q_A(q_\gamma(t))Q_B(q_\gamma(t))
 \right]=tK_{\gamma,t}[A,B].
 \label{eq.fourth-moment}
\end{align}
In particular, since $q_\gamma(t_0)=\sqrt N\,\tilde x_\gamma$, we obtain
\begin{align}
 \E_\gamma\left[
 \langle x_\gamma,Ax_\gamma\rangle
 \langle x_\gamma,Bx_\gamma\rangle
 -\langle\tilde x_\gamma,A\tilde x_\gamma\rangle
 \langle\tilde x_\gamma,B\tilde x_\gamma\rangle
 \right]=t_0N^{-2}K_\gamma[A,B].
 \label{eq.scaled-fourth-moment}
\end{align}
\end{lemma}

\begin{proof}[Proof of Lemma \ref{lem.moment-identities}]
The first identity follows from covariance matching; that is,
\begin{align*}
 \E_\gamma\left[\langle x_\gamma,Ax_\gamma\rangle\right]
 =\E_\gamma\left[\langle\tilde x_\gamma,A\tilde x_\gamma\rangle\right]
 =N^{-1}\Tr(\Sigma A).
\end{align*}
Moreover, by independence and the Gaussian moment formula,
\begin{align*}
 \E_\gamma[Q_A(q_\gamma(t))Q_B(q_\gamma(t))]
 =2\Tr(\Sigma A\Sigma B)
  + K_\gamma^0[\mathcal B_t(A),\mathcal B_t(B)].
\end{align*}
Subtracting this identity from its value at $t=0$ proves the second
identity. The last identity follows by taking $t=t_0$.
\end{proof}

\begin{lemma}
\label{lem.centered-triple}
Let $A,B,C$ be complex symmetric matrices held fixed under
$\E_\gamma$. For $t\in[0,t_*]$, define
\begin{align}
 C_{\gamma,t}(A,B,C)
:=\E_\gamma\bigl[Q_A(q_\gamma)Q_B(q_\gamma)Q_C(q_\gamma)\bigr]
 -\E_\gamma\bigl[
 Q_A(q_\gamma(t))Q_B(q_\gamma(t))Q_C(q_\gamma(t))
 \bigr].
 \label{eq.centered-triple-difference}
\end{align}
We have
\begin{align}
 |C_{\gamma,t}(A,B,C)|
 \le Ct\|A\|_{\HS}\|B\|_{\HS}\|C\|_{\HS}.
 \label{eq.centered-triple-bound}
\end{align}
\end{lemma}

\begin{proof}[Proof of Lemma \ref{lem.centered-triple}]
Set
\begin{align*}
 P(s):=\E_\gamma\bigl[
 Q_A(q_\gamma(s))Q_B(q_\gamma(s))Q_C(q_\gamma(s))
 \bigr],\qquad 0\le s\le t_* \le 1.
\end{align*}
We have $C_{\gamma,t}(A,B,C)=P(0)-P(t)$.
To bound the derivative of $P$, define
\begin{align*}
 R_A(s):=Q_{\mathcal B_s(A)}(q_\gamma),\qquad
 G_A:=g_\gamma^TAg_\gamma-\Tr A,\qquad
 L_A(s):=q_\gamma^TS_sAg_\gamma.
\end{align*}
Thus
\[
 Q_A(q_\gamma(s))=R_A(s)+sG_A+2\sqrt{s}\,L_A(s).
\]
Expanding the product in $P(s)$, terms containing an odd
number of factors $L$ vanish by Gaussian symmetry.
Among the terms with no factors $L$, those containing one or two
factors $G$ vanish by independence and centering.
Thus only terms of the forms $RRR$, $RLL$, $GLL$, and $GGG$ survive,
and evaluating their Gaussian expectations gives
\begin{align*}
 P(s)
 ={}&\E_\gamma[R_A(s)R_B(s)R_C(s)]
 +4s\sum_{\mathrm{cyc}}
   \E_\gamma\bigl[
     R_A(s)Q_{\mathcal D_s(B,C)}(q_\gamma)\bigr]
 \\
 &+8s^2\sum_{\mathrm{cyc}}
    \Tr\bigl(\Sigma S_sBAC S_s\bigr)
 +8s^3\Tr(ABC),
\end{align*}
where the cyclic sum runs over $(A,B,C),(B,C,A),(C,A,B)$ and
\[
 \mathcal D_s(B,C):=\tfrac12S_s(BC+CB)S_s.
\]
Here we used
\begin{align*}
 \E_g[L_B(s)L_C(s)]
 &=q_\gamma^T\mathcal D_s(B,C)q_\gamma,\\
 \E_g[G_A L_B(s)L_C(s)]
 &=2q_\gamma^TS_sBAC S_sq_\gamma,\\
 \E_g[G_AG_BG_C]
 &=8\Tr(ABC),
\end{align*}
together with $\E_\gamma R_A(s)=0$ and
$\E_\gamma q_\gamma q_\gamma^T=\Sigma$.

Since $S_s=(I_M-s\Sigma^{-1})^{1/2}$, we have $S_s'=\frac{\mathrm d}{\mathrm ds}S_s
 =-\frac12\Sigma^{-1}S_s^{-1}$. Differentiating $\mathcal B_s(A)=S_sAS_s$ and
$\mathcal D_s(B,C)=\tfrac12S_s(BC+CB)S_s$ gives
\begin{align*}
 \mathcal B_s'(A)
 &=S_s'AS_s+S_sAS_s',\\
 \mathcal D_s'(B,C)
 &=\tfrac12S_s'(BC+CB)S_s
   +\tfrac12S_s(BC+CB)S_s'.
\end{align*}
The uniform bounds on $S_s$ and $S_s'$ therefore imply
\begin{align*}
 \|\mathcal B_s(A)\|_{\HS}
 +\|\mathcal B_s'(A)\|_{\HS}
 &\le C\|A\|_{\HS},\\
 \|\mathcal D_s(B,C)\|_{\HS}
 +\|\mathcal D_s'(B,C)\|_{\HS}
 &\le C\|B\|_{\HS}\|C\|_{\HS}.
\end{align*}
By H\"older's inequality and \eqref{eq:quadratic-form-moments}
with $p=3$, the derivative of the first term in $P(s)$ is bounded by
$C\|A\|_{\HS}\|B\|_{\HS}\|C\|_{\HS}$.
The second term satisfies the same derivative bound by
\eqref{eq:quadratic-form-moments} with $p=2$.

For the remaining terms, note that
$S_s\Sigma S_s=\Sigma-sI_M$, we have
\[
 \Tr(\Sigma S_sBAC S_s)
 =\Tr(S_s \Sigma S_s BAC )
 =\Tr((\Sigma-sI_M)BAC).
\]
The boundedness of $\Sigma$ therefore bounds this expression and its derivative, as well as
$\Tr(ABC)$, by
$C\|A\|_{\HS}\|B\|_{\HS}\|C\|_{\HS}$.
Thus, we have
\[
 \sup_{s\in[0,t_*]}|P'(s)|
 \le C\|A\|_{\HS}\|B\|_{\HS}\|C\|_{\HS}.
\]
Consequently,
\[
 |C_{\gamma,t}(A,B,C)|
 =|P(0)-P(t)|
 =
\left|\int_0^t P'(s)\,\mathrm ds\right|
 \le t\sup_{s\in[0,t]}|P'(s)|
 \le Ct\|A\|_{\HS}\|B\|_{\HS}\|C\|_{\HS},
\]
which proves \eqref{eq.centered-triple-bound}.
\end{proof}

For notational convenience, we set $R_{\gamma,\sigma}:=\mathcal G_{1,\gamma}^{(\gamma)}(z_\sigma)$. Combining with Lemma~\ref{lem.moment-identities} and \eqref{eq.centered-triple-difference}, we get the following identities.

\begin{proposition}\label{prop.momequation}
For $\sigma,\rho,\nu\in\{+1,-1\}$, we have
\begin{equation*}
\begin{aligned}
\E_\gamma\bigl[
a_\sigma(x_\gamma)-a_\sigma(\tilde x_\gamma)
\bigr]
&=0,
\\
\E_\gamma\bigl[
a_\sigma(x_\gamma)\xi_\sigma(x_\gamma)
-a_\sigma(\tilde x_\gamma)\xi_\sigma(\tilde x_\gamma)
\bigr]
&=t_0N^{-2}
K_\gamma[R_{\gamma,\sigma}^2,R_{\gamma,\sigma}],
\\
\E_\gamma\bigl[
a_\sigma(x_\gamma)\xi_\sigma(x_\gamma)^2
-a_\sigma(\tilde x_\gamma)\xi_\sigma(\tilde x_\gamma)^2
\bigr]
&=N^{-3}C_{\gamma,t_0}
(R_{\gamma,\sigma}^2,R_{\gamma,\sigma},R_{\gamma,\sigma})
\\
&\phantom{={}}+t_0N^{-3}
(\Tr\Sigma R_{\gamma,\sigma}^2)
K_\gamma[R_{\gamma,\sigma},R_{\gamma,\sigma}],
\\
\E_\gamma\bigl[
a_\sigma(x_\gamma)a_\rho(x_\gamma)
-a_\sigma(\tilde x_\gamma)a_\rho(\tilde x_\gamma)
\bigr]
&=t_0N^{-2}
K_\gamma[R_{\gamma,\sigma}^2,R_{\gamma,\rho}^2],
\\
\E_\gamma\bigl[
a_\sigma(x_\gamma)a_\rho(x_\gamma)\xi_\rho(x_\gamma)
-a_\sigma(\tilde x_\gamma)a_\rho(\tilde x_\gamma)
 \xi_\rho(\tilde x_\gamma)
\bigr]
&=N^{-3}C_{\gamma,t_0}
(R_{\gamma,\sigma}^2,R_{\gamma,\rho}^2,R_{\gamma,\rho})
\\
&\phantom{={}}+t_0N^{-3}\Bigl[
(\Tr\Sigma R_{\gamma,\sigma}^2)
K_\gamma[R_{\gamma,\rho}^2,R_{\gamma,\rho}]
\\
&\phantom{={+t_0N^{-3}\Bigl[}}
+(\Tr\Sigma R_{\gamma,\rho}^2)
K_\gamma[R_{\gamma,\sigma}^2,R_{\gamma,\rho}]
\Bigr],
\\
\E_\gamma\bigl[
a_\sigma(x_\gamma)a_\rho(x_\gamma)a_\nu(x_\gamma)
-a_\sigma(\tilde x_\gamma)a_\rho(\tilde x_\gamma)
 a_\nu(\tilde x_\gamma)
\bigr]
&=N^{-3}C_{\gamma,t_0}
(R_{\gamma,\sigma}^2,R_{\gamma,\rho}^2,R_{\gamma,\nu}^2)
\\
&\phantom{={}}+t_0N^{-3}\Bigl[
(\Tr\Sigma R_{\gamma,\sigma}^2)
K_\gamma[R_{\gamma,\rho}^2,R_{\gamma,\nu}^2]
\\
&\phantom{={+t_0N^{-3}\Bigl[}}
+(\Tr\Sigma R_{\gamma,\rho}^2)
K_\gamma[R_{\gamma,\sigma}^2,R_{\gamma,\nu}^2]
\\
&\phantom{={+t_0N^{-3}\Bigl[}}
+(\Tr\Sigma R_{\gamma,\nu}^2)
K_\gamma[R_{\gamma,\sigma}^2,R_{\gamma,\rho}^2]
\Bigr].
\end{aligned}
\end{equation*}
\end{proposition}

\begin{proof}[Proof of Proposition \ref{prop.momequation}]
The first two identities follow from Lemma~\ref{lem.moment-identities}. We only prove the third identity in detail, since the others can be proved similarly. For $v_\gamma\in\{x_\gamma,\tilde x_\gamma\}$, write $u_\gamma=\sqrt N\,v_\gamma$. By definition, we have
\[
 a_\sigma(v_\gamma)
 =N^{-1}\bigl(
 Q_{R_{\gamma,\sigma}^2}(u_\gamma)+\Tr\Sigma R_{\gamma,\sigma}^2
 \bigr),
 \qquad
 \xi_\sigma(v_\gamma)
 =N^{-1}Q_{R_{\gamma,\sigma}}(u_\gamma).
\]
Hence,
\[
 a_\sigma(v_\gamma)\xi_\sigma(v_\gamma)^2
 =N^{-3}Q_{R_{\gamma,\sigma}^2}(u_\gamma)
          Q_{R_{\gamma,\sigma}}(u_\gamma)^2
 +N^{-3}(\Tr\Sigma R_{\gamma,\sigma}^2)
          Q_{R_{\gamma,\sigma}}(u_\gamma)^2.
\]
Using $\sqrt N\,x_\gamma=q_\gamma$ and
$\sqrt N\,\tilde x_\gamma=q_\gamma(t_0)$, we obtain
\begin{align*}
&\E_\gamma\bigl[
 a_\sigma(x_\gamma)\xi_\sigma(x_\gamma)^2
 -a_\sigma(\tilde x_\gamma)\xi_\sigma(\tilde x_\gamma)^2
 \bigr]
\\
&\quad=N^{-3}\E_\gamma\Bigl[
 Q_{R_{\gamma,\sigma}^2}(q_\gamma)
 Q_{R_{\gamma,\sigma}}(q_\gamma)^2
 -
 Q_{R_{\gamma,\sigma}^2}(q_\gamma(t_0))
 Q_{R_{\gamma,\sigma}}(q_\gamma(t_0))^2
 \Bigr]
\\
&\qquad\quad
 +N^{-3}(\Tr\Sigma R_{\gamma,\sigma}^2)\E_\gamma\Bigl[
 Q_{R_{\gamma,\sigma}}(q_\gamma)^2
 -
 Q_{R_{\gamma,\sigma}}(q_\gamma(t_0))^2
 \Bigr]
\\
&\quad=N^{-3}C_{\gamma,t_0}
 (R_{\gamma,\sigma}^2,R_{\gamma,\sigma},R_{\gamma,\sigma})
 +t_0N^{-3}(\Tr\Sigma R_{\gamma,\sigma}^2)
 K_\gamma[R_{\gamma,\sigma},R_{\gamma,\sigma}],
\end{align*}
where the last equality follows from the definition of
$C_{\gamma,t_0}$ and Lemma~\ref{lem.moment-identities}.
\end{proof}

Using Proposition~\ref{prop.momequation} and taking conditional
expectations with respect to $\E_\gamma$, we can rewrite the six terms $\mathcal T_{k,\gamma}$ in \eqref{eq.tauk defi} as
\begin{equation}\label{eq.tauk defi2}
\begin{aligned}
\mathcal T_{1,\gamma}
&=0,
\\
\mathcal T_{2,\gamma}
&=-t_0\eta_0N^{-2}
\E\Biggl[
F'(\mu_\gamma)\sum_{\sigma=\pm1}
h_{1,\gamma,\sigma}
K_\gamma[R_{\gamma,\sigma}^2,R_{\gamma,\sigma}]
\Biggr],
\\
\mathcal T_{3,\gamma}
&=\eta_0N^{-3}
\E\Biggl[
F'(\mu_\gamma)\sum_{\sigma=\pm1}
h_{2,\gamma,\sigma}
\Bigl\{
C_{\gamma,t_0}
(R_{\gamma,\sigma}^2,R_{\gamma,\sigma},R_{\gamma,\sigma})
\\
&\hspace{5em}
+t_0(\Tr\Sigma R_{\gamma,\sigma}^2)
K_\gamma[R_{\gamma,\sigma},R_{\gamma,\sigma}]
\Bigr\}
\Biggr],
\\
\mathcal T_{4,\gamma}
&=-\frac{t_0\eta_0^2}{2}N^{-2}
\E\Biggl[
F''(\mu_\gamma)\sum_{\sigma,\rho=\pm1}
h_{0,\gamma,\sigma}h_{0,\gamma,\rho}
K_\gamma[R_{\gamma,\sigma}^2,R_{\gamma,\rho}^2]
\Biggr],
\\
\mathcal T_{5,\gamma}
&=\eta_0^2N^{-3}
\E\Biggl[
F''(\mu_\gamma)\sum_{\sigma,\rho=\pm1}
h_{0,\gamma,\sigma}h_{1,\gamma,\rho}
\Bigl\{
C_{\gamma,t_0}
(R_{\gamma,\sigma}^2,R_{\gamma,\rho}^2,R_{\gamma,\rho})
\\
&\hspace{5em}
+t_0\Bigl[
(\Tr\Sigma R_{\gamma,\sigma}^2)
K_\gamma[R_{\gamma,\rho}^2,R_{\gamma,\rho}]
+(\Tr\Sigma R_{\gamma,\rho}^2)
K_\gamma[R_{\gamma,\sigma}^2,R_{\gamma,\rho}]
\Bigr]
\Bigr\}
\Biggr],
\\
\mathcal T_{6,\gamma}
&=\frac{\eta_0^3}{6}N^{-3}
\E\Biggl[
F'''(\mu_\gamma)\sum_{\sigma,\rho,\nu=\pm1}
h_{0,\gamma,\sigma}h_{0,\gamma,\rho}h_{0,\gamma,\nu}
\Bigl\{
C_{\gamma,t_0}
(R_{\gamma,\sigma}^2,R_{\gamma,\rho}^2,R_{\gamma,\nu}^2)
\\
&\hspace{5em}
+t_0\Bigl[
(\Tr\Sigma R_{\gamma,\sigma}^2)
K_\gamma[R_{\gamma,\rho}^2,R_{\gamma,\nu}^2]
+(\Tr\Sigma R_{\gamma,\rho}^2)
K_\gamma[R_{\gamma,\sigma}^2,R_{\gamma,\nu}^2]
\\
&\hspace{8em}
+(\Tr\Sigma R_{\gamma,\nu}^2)
K_\gamma[R_{\gamma,\sigma}^2,R_{\gamma,\rho}^2]
\Bigr]
\Bigr\}
\Biggr].
\end{aligned}
\end{equation}
In \eqref{eq.tauk defi2}, we denote $\mathcal T_{k,\gamma}$'s parts containing $C_{\gamma,t_0}$ and $K_\gamma$ by
$\mathcal T_{k,\gamma}^{(C)}$ and
$\mathcal T_{k,\gamma}^{(K)}$, respectively. In the next section, we will use Lemma~\ref{lem.centered-triple} to estimate $\mathcal T_{k,\gamma}^{(C)}$ and Lemma~\ref{lem.weighted-cumulant} below to estimate $\mathcal T_{k,\gamma}^{(K)}$. The proof of Lemma~\ref{lem.weighted-cumulant} is postponed to Section~\ref{sec.weighted-cumulant}.

\begin{lemma}
\label{lem.weighted-cumulant}
Fix $\gamma\in\{1,\ldots,N\}$ and let $K_\gamma=K_{\gamma,t_0}$, $R_a = \mathcal G_{1,\gamma}^{(\gamma)}(z_a)$, $a=1,2$, where $z_1,z_2$ satisfy \eqref{eq:domain1}.
Let $W_\gamma$ be a random weight measurable with respect
to the columns of $X_\gamma^{(\gamma)}$.
Assume that, for each $\ell\ne\gamma$, there exists
a weight $W_\gamma^{(\ell)}$ measurable with respect
to the columns of $X_\gamma^{(\gamma,\ell)}$ such that,
for every fixed $p\ge1$,
\begin{equation}\label{eq.weight-assumptions}
 \|W_\gamma\|_{p}
 +\|W_\gamma^{(\ell)}\|_{p}
 \le C_pN^{C_W\epsilon},
 \qquad
 \|W_\gamma-W_\gamma^{(\ell)}\|_{p}
 \le C_pN^{-1/3+C_W\epsilon},
\end{equation}
uniformly in $\gamma$ and $\ell$, where $C_W>0$
is a fixed constant.
Then, uniformly in $\gamma$, we have
\begin{align}
 \bigl|\E\bigl[
 W_\gamma K_\gamma[R_1,R_2]
 \bigr]\bigr|
 &\le CN^{7/6+C\epsilon},
 \label{eq.weighted-K11}
 \\
 \bigl|\E\bigl[
 W_\gamma K_\gamma[R_1^2,R_2]
 \bigr]\bigr|
 &\le CN^{11/6+C\epsilon},
 \label{eq.weighted-K21}
 \\
 \bigl|\E\bigl[
 W_\gamma K_\gamma[R_1^2,R_2^2]
 \bigr]\bigr|
 &\le CN^{5/2+C\epsilon}.
 \label{eq.weighted-K22}
\end{align}
\end{lemma}

\subsection{Proof of Lemma~\ref{telescoping error lemma}}
\label{sec.telescoping error lemma}
We now proceed with the proof of Lemma \ref{telescoping error lemma}. Recall \eqref{eq.six-term-comparison}, it suffices to estimate the six terms $\mathcal T_{k,\gamma}$ uniformly in $\gamma$. The decomposition obtained in Section~\ref{sec.comparison-term-analysis} gives
\[
 \mathcal T_{1,\gamma}=0,\qquad
 \mathcal T_{k,\gamma}=\mathcal T_{k,\gamma}^{(K)},
 \quad k=2,4,\qquad
 \mathcal T_{k,\gamma}
 =\mathcal T_{k,\gamma}^{(C)}
  +\mathcal T_{k,\gamma}^{(K)},
 \quad k=3,5,6.
\]
We then estimate $\mathcal T_j^{(C)}$ and $\mathcal T_j^{(K)}$ using Lemma~\ref{lem.centered-triple} and Lemma~\ref{lem.weighted-cumulant}, respectively. 

\begin{proof}[Proof of Lemma~\ref{telescoping error lemma}]
We first consider the terms involving $C_{\gamma,t_0}$.
By \eqref{eq.tauk defi2}, these terms are
\begin{equation*}
\begin{aligned}
\mathcal T_{3,\gamma}^{(C)}
&=\eta_0N^{-3}
\E\Biggl[
F'(\mu_\gamma)\sum_{\sigma=\pm1}
h_{2,\gamma,\sigma}
C_{\gamma,t_0}
(R_{\gamma,\sigma}^2,R_{\gamma,\sigma},R_{\gamma,\sigma})
\Biggr],
\\
\mathcal T_{5,\gamma}^{(C)}
&=\eta_0^2N^{-3}
\E\Biggl[
F''(\mu_\gamma)\sum_{\sigma,\rho=\pm1}
h_{0,\gamma,\sigma}h_{1,\gamma,\rho}
C_{\gamma,t_0}
(R_{\gamma,\sigma}^2,R_{\gamma,\rho}^2,R_{\gamma,\rho})
\Biggr],
\\
\mathcal T_{6,\gamma}^{(C)}
&=\frac{\eta_0^3}{6}N^{-3}
\E\Biggl[
F'''(\mu_\gamma)\sum_{\sigma,\rho,\nu=\pm1}
h_{0,\gamma,\sigma}h_{0,\gamma,\rho}h_{0,\gamma,\nu}
C_{\gamma,t_0}
(R_{\gamma,\sigma}^2,R_{\gamma,\rho}^2,R_{\gamma,\nu}^2)
\Biggr].
\end{aligned}
\end{equation*}
The bounds on Lemma~\ref{lem:res} and the Ward identity give,
for every fixed $p\ge1$,
\begin{equation*}\label{eq.resolvent-HS-moments}
\begin{aligned}
 \bigl\|\|R_{\gamma,\sigma}\|_{\mathrm{HS}}\bigr\|_{p}\le C_pN^{2/3+C\epsilon},
 \quad
 \bigl\|\|R_{\gamma,\sigma}^2\|_{\mathrm{HS}}\bigr\|_{p}\le C_pN^{4/3+C\epsilon}.
\end{aligned}
\end{equation*}
Together with the moment bounds
\[
 \|F^{(j)}(\mu_\gamma)\|_{p}
 +\|h_{k,\gamma,\sigma}\|_{p}
 \le C_{F,p}N^{C\epsilon},
 \qquad j=1,2,3,\quad k=0,1,2,
\]
Lemma~\ref{lem.centered-triple} and H\"older's inequality yield
\begin{align*}
 \bigl|\mathcal T_{3,\gamma}^{(C)}\bigr|
 &\le C_F\eta_0N^{-3}t_0
 N^{4/3+2/3+2/3+C\epsilon}
 \le C_Ft_0N^{-1+C\epsilon},
\\
 \bigl|\mathcal T_{5,\gamma}^{(C)}\bigr|
 &\le C_F\eta_0^2N^{-3}t_0
 N^{4/3+4/3+2/3+C\epsilon}
 \le C_Ft_0N^{-1+C\epsilon},
\\
 \bigl|\mathcal T_{6,\gamma}^{(C)}\bigr|
 &\le C_F\eta_0^3N^{-3}t_0
 N^{4/3+4/3+4/3+C\epsilon}
 \le C_Ft_0N^{-1+C\epsilon}.
\end{align*}
Since $t_0=N^{-1/3+\epsilon}$, we conclude that,
uniformly in $\gamma$,
\begin{equation}\label{eq.C-term-bound}
 \bigl|\mathcal T_{3,\gamma}^{(C)}\bigr|
 +\bigl|\mathcal T_{5,\gamma}^{(C)}\bigr|
 +\bigl|\mathcal T_{6,\gamma}^{(C)}\bigr|
 \le C_FN^{-4/3+C\epsilon}.
\end{equation}

We next consider the terms involving $K_\gamma$. By \eqref{eq.tauk defi2}, these terms are
\begin{equation*}\label{eq.K-terms}
\begin{aligned}
\mathcal T_{2,\gamma}^{(K)}
&=-t_0\eta_0N^{-2}
\E\Biggl[
F'(\mu_\gamma)\sum_{\sigma=\pm1}
h_{1,\gamma,\sigma}
K_\gamma[R_{\gamma,\sigma}^2,R_{\gamma,\sigma}]
\Biggr],
\\
\mathcal T_{3,\gamma}^{(K)}
&=t_0\eta_0N^{-3}
\E\Biggl[
F'(\mu_\gamma)\sum_{\sigma=\pm1}
h_{2,\gamma,\sigma}
(\Tr\Sigma R_{\gamma,\sigma}^2)
K_\gamma[R_{\gamma,\sigma},R_{\gamma,\sigma}]
\Biggr],
\\
\mathcal T_{4,\gamma}^{(K)}
&=-\frac{t_0\eta_0^2}{2}N^{-2}
\E\Biggl[
F''(\mu_\gamma)\sum_{\sigma,\rho=\pm1}
h_{0,\gamma,\sigma}h_{0,\gamma,\rho}
K_\gamma[R_{\gamma,\sigma}^2,R_{\gamma,\rho}^2]
\Biggr],
\\
\mathcal T_{5,\gamma}^{(K)}
&=t_0\eta_0^2N^{-3}
\E\Biggl[
F''(\mu_\gamma)\sum_{\sigma,\rho=\pm1}
h_{0,\gamma,\sigma}h_{1,\gamma,\rho}
\Bigl[
(\Tr\Sigma R_{\gamma,\sigma}^2)
K_\gamma[R_{\gamma,\rho}^2,R_{\gamma,\rho}]
\\
&\hspace{8em}
+(\Tr\Sigma R_{\gamma,\rho}^2)
K_\gamma[R_{\gamma,\sigma}^2,R_{\gamma,\rho}]
\Bigr]
\Biggr],
\\
\mathcal T_{6,\gamma}^{(K)}
&=\frac{t_0\eta_0^3}{6}N^{-3}
\E\Biggl[
F'''(\mu_\gamma)\sum_{\sigma,\rho,\nu=\pm1}
h_{0,\gamma,\sigma}h_{0,\gamma,\rho}h_{0,\gamma,\nu}
\Bigl[
(\Tr\Sigma R_{\gamma,\sigma}^2)
K_\gamma[R_{\gamma,\rho}^2,R_{\gamma,\nu}^2]
\\
&\hspace{8em}
+(\Tr\Sigma R_{\gamma,\rho}^2)
K_\gamma[R_{\gamma,\sigma}^2,R_{\gamma,\nu}^2]
+(\Tr\Sigma R_{\gamma,\nu}^2)
K_\gamma[R_{\gamma,\sigma}^2,R_{\gamma,\rho}^2]
\Bigr]
\Biggr].
\end{aligned}
\end{equation*}
Before applying Lemma~\ref{lem.weighted-cumulant}, we introduce the normalized traces
\begin{align*}
 \theta_{\gamma,\sigma}
 :=N^{-4/3}\Tr\Sigma R_{\gamma,\sigma}^2,
 \qquad
 \theta_{\gamma,\sigma}^{(\ell)}
 :=N^{-4/3}\Tr\Sigma\bigl(R_{\gamma,\sigma}^{(\ell)}\bigr)^2,
 \qquad \ell\ne\gamma,
\end{align*}
where
$R_{\gamma,\sigma}^{(\ell)}
:=\mathcal G_{1,\gamma}^{(\gamma,\ell)}(z_\sigma)$. So each of the term $\mathcal T_{j,\gamma}^{(K)}$ can be written in the form
\begin{align*}
 \text{deterministic coefficient}
 \times \E\bigl[W_\gamma K_\gamma[A,B]\bigr].
\end{align*}
The complete correspondence is listed below.
All signs $\rho,\sigma,\nu$ take values in $\{+1,-1\}$, and each row
is summed over its relevant signs.

\begin{center}
\renewcommand{\arraystretch}{1.6}
\resizebox{\textwidth}{!}{%
\begin{tabular}{c c c c}
\hline
Term & Deterministic coefficient
& Weight $W_\gamma$ & $K_\gamma[A,B]$ \\
\hline
$\mathcal T_{2,\gamma}^{(K)}$
& $-t_0\eta_0N^{-2}$
& $F'(\mu_\gamma)h_{1,\gamma,\sigma}$
& $K_\gamma[R_{\gamma,\sigma}^2,R_{\gamma,\sigma}]$
\\
$\mathcal T_{3,\gamma}^{(K)}$
& $t_0\eta_0N^{-5/3}$
& $F'(\mu_\gamma)h_{2,\gamma,\sigma}
   \theta_{\gamma,\sigma}$
& $K_\gamma[R_{\gamma,\sigma},R_{\gamma,\sigma}]$
\\
$\mathcal T_{4,\gamma}^{(K)}$
& $-\frac12t_0\eta_0^2N^{-2}$
& $F''(\mu_\gamma)
   h_{0,\gamma,\sigma}h_{0,\gamma,\rho}$
& $K_\gamma[R_{\gamma,\sigma}^2,R_{\gamma,\rho}^2]$
\\
$\mathcal T_{5,\gamma}^{(K)}$, first term
& $t_0\eta_0^2N^{-5/3}$
& $F''(\mu_\gamma)
   h_{0,\gamma,\sigma}h_{1,\gamma,\rho}
   \theta_{\gamma,\sigma}$
& $K_\gamma[R_{\gamma,\rho}^2,R_{\gamma,\rho}]$
\\
$\mathcal T_{5,\gamma}^{(K)}$, second term
& $t_0\eta_0^2N^{-5/3}$
& $F''(\mu_\gamma)
   h_{0,\gamma,\sigma}h_{1,\gamma,\rho}
   \theta_{\gamma,\rho}$
& $K_\gamma[R_{\gamma,\sigma}^2,R_{\gamma,\rho}]$
\\
$\mathcal T_{6,\gamma}^{(K)}$, first term
& $\frac16t_0\eta_0^3N^{-5/3}$
& $F'''(\mu_\gamma)
   h_{0,\gamma,\sigma}h_{0,\gamma,\rho}h_{0,\gamma,\nu}
   \theta_{\gamma,\sigma}$
& $K_\gamma[R_{\gamma,\rho}^2,R_{\gamma,\nu}^2]$
\\
$\mathcal T_{6,\gamma}^{(K)}$, second term
& $\frac16t_0\eta_0^3N^{-5/3}$
& $F'''(\mu_\gamma)
   h_{0,\gamma,\sigma}h_{0,\gamma,\rho}h_{0,\gamma,\nu}
   \theta_{\gamma,\rho}$
& $K_\gamma[R_{\gamma,\sigma}^2,R_{\gamma,\nu}^2]$
\\
$\mathcal T_{6,\gamma}^{(K)}$, third term
& $\frac16t_0\eta_0^3N^{-5/3}$
& $F'''(\mu_\gamma)
   h_{0,\gamma,\sigma}h_{0,\gamma,\rho}h_{0,\gamma,\nu}
   \theta_{\gamma,\nu}$
& $K_\gamma[R_{\gamma,\sigma}^2,R_{\gamma,\rho}^2]$
\\
\hline
\end{tabular}
}
\end{center}

For each weight $W_\gamma$ listed above, we also define $W_\gamma^{(\ell)}$ by replacing $\mu_\gamma$, $h_{k,\gamma,\sigma}$ and $\theta_{\gamma,\sigma}$ with their counterpart after deleting the additional column
$\ell\ne\gamma$.
In particular, $\mu_\gamma^{(\ell)}$ and
$h_{k,\gamma,\sigma}^{(\ell)}$ are obtained by replacing
$R_{\gamma,\sigma}$ with $R_{\gamma,\sigma}^{(\ell)}$
in their definitions, while keeping the deterministic
terms unchanged.
Thus $W_\gamma^{(\ell)}$ is measurable with respect
to the columns of $X_\gamma^{(\gamma,\ell)}$.

We first verify the assumptions of
Lemma~\ref{lem.weighted-cumulant}.
Bounds on Lemma~\ref{lem:res}, the Sherman Morrison formula
and the Ward identity give, for every fixed $p\ge1$,
\begin{align*}
 \|\mu_\gamma\|_p
 +\|\mu_\gamma^{(\ell)}\|_p
 &\le C_pN^{C\epsilon},
 &
 \|\mu_\gamma-\mu_\gamma^{(\ell)}\|_p
 &\le C_pN^{-1/3+C\epsilon},
 \\
 \|h_{k,\gamma,\sigma}\|_p
 +\|h_{k,\gamma,\sigma}^{(\ell)}\|_p
 &\le C_pN^{C\epsilon},
 &
 \|h_{k,\gamma,\sigma}
      -h_{k,\gamma,\sigma}^{(\ell)}\|_p
 &\le C_pN^{-2/3+C\epsilon},
 \\
 \|\theta_{\gamma,\sigma}\|_p
 +\|\theta_{\gamma,\sigma}^{(\ell)}\|_p
 &\le C_pN^{C\epsilon},
 &
 \|\theta_{\gamma,\sigma}
      -\theta_{\gamma,\sigma}^{(\ell)}\|_p
 &\le C_pN^{-1/3+C\epsilon},
\end{align*}
uniformly in their indices, where $k=0,1,2$.
Moreover, the regularity assumptions on $F$ and
the mean value theorem imply, for $j=1,2,3$,
\begin{align*}
 \|F^{(j)}(\mu_\gamma)\|_p
 +\|F^{(j)}(\mu_\gamma^{(\ell)})\|_p
 \le C_{F,p}N^{C\epsilon},
 \quad 
 \|F^{(j)}(\mu_\gamma)
      -F^{(j)}(\mu_\gamma^{(\ell)})\|_p
 \le C_{F,p}N^{-1/3+C\epsilon}.
\end{align*}
Hence all weights in the table satisfy the assumptions
of Lemma~\ref{lem.weighted-cumulant}. Applying Lemma~\ref{lem.weighted-cumulant} to each row of the table yields
\begin{align*}
 \bigl|\mathcal T_{2,\gamma}^{(K)}\bigr|
 &\le C_Ft_0\eta_0N^{-2}
 N^{11/6+C\epsilon}
 \le C_Ft_0N^{-5/6+C\epsilon},
 \\
 \bigl|\mathcal T_{3,\gamma}^{(K)}\bigr|
 &\le C_Ft_0\eta_0N^{-5/3}
 N^{7/6+C\epsilon}
 \le C_Ft_0N^{-7/6+C\epsilon},
 \\
 \bigl|\mathcal T_{4,\gamma}^{(K)}\bigr|
 &\le C_Ft_0\eta_0^2N^{-2}
 N^{5/2+C\epsilon}
 \le C_Ft_0N^{-5/6+C\epsilon},
 \\
 \bigl|\mathcal T_{5,\gamma}^{(K)}\bigr|
 &\le C_Ft_0\eta_0^2N^{-5/3}
 N^{11/6+C\epsilon}
 \le C_Ft_0N^{-7/6+C\epsilon},
 \\
 \bigl|\mathcal T_{6,\gamma}^{(K)}\bigr|
 &\le C_Ft_0\eta_0^3N^{-5/3}
 N^{5/2+C\epsilon}
 \le C_Ft_0N^{-7/6+C\epsilon},
\end{align*}
where we use $\eta_0=N^{-2/3-\epsilon}$.
Substituting $t_0=N^{-1/3+\epsilon}$ gives
\begin{equation*}\label{eq.K-term-bound}
\begin{aligned}
 \bigl|\mathcal T_{2,\gamma}^{(K)}\bigr|
 +\bigl|\mathcal T_{4,\gamma}^{(K)}\bigr|
 \le C_FN^{-7/6+C\epsilon},
 \quad
 \bigl|\mathcal T_{3,\gamma}^{(K)}\bigr|
 +\bigl|\mathcal T_{5,\gamma}^{(K)}\bigr|
 +\bigl|\mathcal T_{6,\gamma}^{(K)}\bigr|
 \le C_FN^{-3/2+C\epsilon}.
\end{aligned}
\end{equation*}
It follow that, uniformly in $\gamma$, we have
\begin{align}
 \sum_{k=2}^{6}
 \bigl|\mathcal T_{k,\gamma}^{(K)}\bigr|
 \le C_FN^{-7/6+C\epsilon},
\end{align}

Combining this estimate with the bounds for
$\mathcal T_{k,\gamma}^{(C)}$ and
$\mathcal T_{1,\gamma}=0$, we obtain from
\eqref{eq.six-term-comparison}
\begin{align*}
 &\left|
 \E F\bigl(\eta_0\Im\Tr\mathcal G_{2,\gamma}(z)\bigr)
 -
 \E F\bigl(\eta_0\Im\Tr\mathcal G_{2,\gamma-1}(z)\bigr)
 \right|
 \\
 &\qquad\le
 \sum_{k=2}^{6}\bigl|\mathcal T_{k,\gamma}^{(K)}\bigr|
 +\sum_{k\in\{3,5,6\}}
       \bigl|\mathcal T_{k,\gamma}^{(C)}\bigr|
 +CN^{-4/3+C\epsilon}
 \\
 &\qquad\le
 C_FN^{-7/6+C\epsilon}
 +C_FN^{-4/3+C\epsilon}
 \le C_FN^{-7/6+C\epsilon},
\end{align*}
uniformly in $\gamma$. This completes the proof of Lemma~\ref{telescoping error lemma}.
\end{proof}

\section{Proof of Lemma~\ref{lem.weighted-cumulant}}
\label{sec.weighted-cumulant}
\subsection{Properties of the fourth-cumulant quantity}
\label{sec.cumulant-properties}
Fixed the column $\gamma$. Recall that we write $K_{\gamma} =  K_{\gamma,t_0}$ and $(K_{\gamma})_{abcd}=(K_{\gamma,t_0})_{abcd}$, with
\begin{align*}
 K_{\gamma}[A,B]:=\sum_{a,b,c,d}(K_{\gamma})_{abcd}A_{ab}B_{cd},
 \quad
 (K_{\gamma})_{abcd}
 :=t_0^{-1}\left\{
 \kappa_4(q_\gamma)_{abcd}
 -\kappa_4(q_\gamma(t_0))_{abcd}\right\},
\end{align*}
where $q_\gamma(t)=(I_M-t\Sigma^{-1})^{1/2}q_\gamma
+\sqrt{t}\,g_\gamma$ and for $u \in \{q_\gamma,q_\gamma(t_0)\}$
\begin{align*}
 \kappa_4(u)_{abcd}
 =\E[u_au_bu_cu_d]
 -\Sigma_{ab}\Sigma_{cd}
 -\Sigma_{ac}\Sigma_{bd}
 -\Sigma_{ad}\Sigma_{bc}.
\end{align*}
The tensor $(K_{\gamma})_{abcd}$ is fully symmetric, meaning that it is invariant under any permutation of the indices $a,b,c,d$. We first establish some properties of $K_{\gamma}[A,B]$ that will be used in the proof of
Lemma~\ref{lem.weighted-cumulant}. We begin with a Hilbert--Schmidt bound of $K_{\gamma}[A,B]$ and then estimate the $U_\ell$ and $V_\ell$ arising from
the deletion of one additional column.
\begin{lemma}\label{lem.K-HS-bound}
For any complex symmetric matrices $A,B$, we have
\begin{equation}\label{eq.K-HS-bound}
 |K_\gamma[A,B]|\le C\|A\|_{\HS}\|B\|_{\HS},
\end{equation}
where constant $C \ge 0$ is independent of $N$ and $\gamma$.
\end{lemma}

\begin{proof}
Set $S_s=(I_M-s\Sigma^{-1})^{1/2}$ and $\mathcal B_t(A)=S_tAS_t$. By \eqref{eq:quadratic-form-moments} and the boundedness of $\Sigma$,
\[
 |K_\gamma^0[A,B]|
 =|\E_\gamma[Q_A(q_\gamma)Q_B(q_\gamma)]-2\Tr(\Sigma A\Sigma B)|
 \le C\|A\|_{\HS}\|B\|_{\HS}.
\]
uniformly in $N$ and $\gamma$. Note that
\[
 K_\gamma[A,B]
 =K_{\gamma,t_0}[A,B]
 =\frac1{t_0}\left(
 K_\gamma^0[S_0AS_0,S_0BS_0]
 -K_\gamma^0[S_{t_0}AS_{t_0},S_{t_0}BS_{t_0}]
 \right).
\]
To estimate this difference, we define
\[
 F(s):=K_\gamma^0[S_sAS_s,S_sBS_s],
 \qquad 0\le s\le t_0.
\]
Since $t_0\le\tau/2$, the lower and upper bounds on $\Sigma$ imply
\[
 S_s'=-\frac12\Sigma^{-1}S_s^{-1},
 \qquad
 \|S_s\|_{\op}+\|S_s'\|_{\op}\le C,
 \qquad
 0\le s\le t_0.
\]
Then differentiating $F(s)$ gives,
\begin{align*}
F'(s)=K_\gamma^0[S_s'AS_s+S_sAS_s',S_sBS_s]+K_\gamma^0[S_sAS_s,S_s'BS_s+S_sBS_s'].
\end{align*}
Using the bound on $K_\gamma^0$ together with
\begin{align*}
 \|S_sAS_s\|_{\HS}
 &\le \|S_s\|_{\op}^2\|A\|_{\HS}
 \le C\|A\|_{\HS},\\
 \|S_s'AS_s+S_sAS_s'\|_{\HS}
 &\le 2\|S_s'\|_{\op}\|S_s\|_{\op}\|A\|_{\HS}
 \le C\|A\|_{\HS},
\end{align*}
and the analogous estimates for $B$, we obtain
\[
 \sup_{0\le s\le t_0}|F'(s)|
 \le C\|A\|_{\HS}\|B\|_{\HS}.
\]
Consequently,
\[
 |K_\gamma[A,B]|
 =\left|-\frac1{t_0}\int_0^{t_0}F'(s)\,\mathrm ds\right|
 \le C\|A\|_{\HS}\|B\|_{\HS},
\]
which proves the claim.
\end{proof}

It is also useful to introduce the linear map
\[
 (\cK_\gamma(B))_{ab}:=\sum_{c,d}(K_\gamma)_{abcd}B_{cd}.
\]
The matrix $\cK_\gamma(B)$ is symmetric and satisfies, uniformly in $\gamma$,
\begin{equation}\label{eq.K-map-bound}
 \|\cK_\gamma(B)\|_{\HS}\le C\|B\|_{\HS}.
\end{equation}
Indeed, we have
\begin{align*}
 \|\mathcal K_\gamma(B)\|_{\HS}^2
 &=\sum_{a,b}
 \overline{(\mathcal K_\gamma(B))_{ab}}(\mathcal K_\gamma(B))_{ab}
 =\sum_{a,b,c,d}(K_\gamma)_{abcd}
   \overline{(\mathcal K_\gamma(B))_{ab}}B_{cd}\\
 &=K_\gamma[\overline{\mathcal K_\gamma(B)},B]
 \le C\|\mathcal K_\gamma(B)\|_{\HS}\|B\|_{\HS}.
\end{align*}

We next introduce two quantities that arise from the resolvent
expansions obtained by removing an additional column in the proof of
Lemma~\ref{lem.weighted-cumulant}. Set
\begin{equation}\label{eq.population-stability-matrix}
 \begin{gathered}
 \Pi(z):=-\frac1z(I_M+m_2(z)\Sigma)^{-1},\qquad
 d_0(z):=1+\frac1N\Tr(\Sigma\Pi(z))=-\frac1{zm_2(z)},\\
 D(z):=\left(zI_M-d_0(z)^{-1}\Sigma\right)^{-1}
       =\frac1z(I_M+m_2(z)\Sigma)^{-1}=-\Pi(z).
 \end{gathered}
\end{equation}

The Assumption~\ref{assu.population} ensures that the right edge $\lambda_+$ is regular. This yields, uniformly in \eqref{eq:domain1},
\begin{equation}\label{eq.population-stability-bound}
 |d_0(z)|+|d_0(z)^{-1}|+\|D(z)\|_{\op}\le C.
\end{equation}
In particular, $D(z)$ is a deterministic symmetric matrix commuting
with $\Sigma$.

Fix $\gamma$ and delete one additional column $\ell\ne\gamma$. We write
\[
 A_\ell:=\mathcal G_{1,\gamma}^{(\gamma,\ell)}(z_1),\qquad
 B_\ell:=\mathcal G_{1,\gamma}^{(\gamma,\ell)}(z_2),\qquad
 D:=D(z_1),
\]
and let $v_\ell$ denote column $\ell$ of $X_\gamma$(so $v_\ell=\tilde x_\ell$ or $v_\ell=x_\ell$ ). Thus
$\E_\ell v_\ell v_\ell^T=N^{-1}\Sigma$. Define
\[
 U_\ell:=K_\gamma[Dv_\ell v_\ell^TA_\ell,B_\ell],\qquad
 V_\ell:=K_\gamma[Dv_\ell v_\ell^TA_\ell,
                   B_\ell v_\ell v_\ell^TB_\ell].
\]
Recall that $\E_\ell$ denotes expectation with respect to $(x_\ell,x_\ell^w)$. In particular, $A_\ell$ and $B_\ell$ are fixed under
$\E_\ell$ and independent of $v_\ell$. We then establish the following estimates for $U_\ell$ and $V_\ell$ .

\begin{lemma}\label{lem.internal-column}
Under the preceding assumptions, for every fixed $p\ge1$, we have
\[
 \E_\ell U_\ell=N^{-1}K_\gamma[D\Sigma A_\ell,B_\ell],
\]
and
\[
 \|U_\ell\|_p\le C_pN^{1/3+C\epsilon},\qquad
 \|V_\ell\|_p\le C_pN^{1/2+C\epsilon},\qquad
 \|\E_\ell V_\ell\|_p\le C_pN^{C\epsilon},
\]
uniformly in $\gamma$, $\ell\ne\gamma$, and the admissible spectral
parameters.
\end{lemma}

\begin{proof}
We first consider $U_\ell$. Set $M_\ell:=\cK_\gamma(B_\ell)$. By definition, we have
\begin{align*}
 K_\gamma[Dv_\ell v_\ell^TA_\ell,B_\ell]
 &=\sum_{a,b,e,f}
 D_{ae}(v_\ell)_e(v_\ell)_f(A_\ell)_{fb}
 \left(\sum_{c,d}(K_\gamma)_{abcd}(B_\ell)_{cd}\right)\\
 &=\sum_{a,b,e,f}
 D_{ae}(v_\ell)_e(v_\ell)_f(A_\ell)_{fb}(M_\ell)_{ab}\\
 &=\sum_{e,f}(v_\ell)_e(DM_\ell A_\ell)_{ef}(v_\ell)_f
 =v_\ell^TDM_\ell A_\ell v_\ell.
\end{align*}
Here we used the symmetry of $D$ and $A_\ell$.
Using the symmetry of $M_\ell$ instead gives
\[
 U_\ell
 =v_\ell^TA_\ell M_\ell Dv_\ell
 =v_\ell^T
 \frac{DM_\ell A_\ell+A_\ell M_\ell D}{2}
 v_\ell.
\]
In particular, since $A_\ell$, $M_\ell$, and $D$ are fixed
under $\E_\ell$, and
$\E_\ell[(v_\ell)_e(v_\ell)_f]=N^{-1}\Sigma_{ef}$,
\begin{align*}
 \E_\ell U_\ell
 =\sum_{e,f}(A_\ell M_\ell D)_{ef}
   \E_\ell[(v_\ell)_e(v_\ell)_f]
 =N^{-1}\Tr(A_\ell M_\ell D\Sigma)
 =N^{-1}K_\gamma[D\Sigma A_\ell,B_\ell].
\end{align*}
Moreover, \eqref{eq.K-HS-bound},
\eqref{eq.population-stability-bound}, and Lemma~\ref{lem:res} give
\begin{equation}\label{eq.EiUiestimate1}
 \|\E_\ell U_\ell\|_p
 \le CN^{-1}
   \bigl\|\|A_\ell\|_{\HS}\|B_\ell\|_{\HS}\bigr\|_p
 \le C_pN^{1/3+C\epsilon}.
\end{equation}
Applying \eqref{eq:quadratic-form-moments} on Assumption~\ref{assu.1.1}, followed by \eqref{eq.K-map-bound}, we obtain
\begin{equation}\label{eq.EiUiestimate2}
\begin{aligned}
 \|U_\ell-\E_\ell U_\ell\|_p
 &\le C_pN^{-1}
 \left\|
 \left\|
 \frac{A_\ell M_\ell D+DM_\ell A_\ell}{2}
 \right\|_{\HS}
 \right\|_p\\
 &\le C_pN^{-1}
 \bigl\|\|A_\ell\|_{\op}\|M_\ell\|_{\HS}\bigr\|_p\\
 &\le C_pN^{-1}
 \bigl\|\|A_\ell\|_{\op}\|B_\ell\|_{\HS}\bigr\|_p
 \le C_pN^{1/3+C\epsilon}.
\end{aligned}
\end{equation}
Combining \eqref{eq.EiUiestimate1} and \eqref{eq.EiUiestimate2}, we finally get
\[
 \|U_\ell\|_p
 \le\|U_\ell-\E_\ell U_\ell\|_p
    +\|\E_\ell U_\ell\|_p
 \le C_pN^{1/3+C\epsilon}.
\]

We next consider $V_\ell$. Similarly,
\[
 |V_\ell|
 \le C\|Dv_\ell v_\ell^TA_\ell\|_{\HS}
          \|B_\ell v_\ell v_\ell^TB_\ell\|_{\HS}
 \le C|v_\ell|\,|A_\ell v_\ell|\,|B_\ell v_\ell|^2,
\]
so H\"older's inequality and Lemma~\ref{lem:res} yield
$\|V_\ell\|_p\le C_pN^{1/2+C\epsilon}$.

It remains to estimate $\E_\ell V_\ell$. Denote the fourth cumulant
tensor of $\sqrt N\,v_\ell$ by $C^{(\ell,\gamma)}$. More explicitly,
\[
 C^{(\ell,\gamma)}=
 \begin{cases}
  \kappa_4(q_\ell(t_0)),&\ell<\gamma,\\
  \kappa_4(q_\ell),&\ell>\gamma.
 \end{cases}
\]
By definition, we have
\[
 N^2\E_\ell[(v_\ell)_e(v_\ell)_f(v_\ell)_g(v_\ell)_h]
 =\Sigma_{ef}\Sigma_{gh}+\Sigma_{eg}\Sigma_{fh}
  +\Sigma_{eh}\Sigma_{fg}+C^{(\ell,\gamma)}_{efgh}.
\]
Contracting the three covariance pairings and using the full symmetry of $K_\gamma$, we obtain
\begin{equation}\label{eq.EkVkall}
 \E_\ell V_\ell=N^{-2}\Bigl(
 K_\gamma[D\Sigma A_\ell,B_\ell\Sigma B_\ell]
 +2K_\gamma[D\Sigma B_\ell,A_\ell\Sigma B_\ell]
 +\Theta_{\gamma,\ell}(D)\Bigr),
\end{equation}
where
\[
 \Theta_{\gamma,\ell}(D):=
 \sum_{a,b,c,d,e,f,g,h}
 (K_\gamma)_{abcd}D_{ae}C^{(\ell,\gamma)}_{efgh}
 (A_\ell)_{fb}(B_\ell)_{cg}(B_\ell)_{hd}.
\]
To estimate $\Theta_{\gamma,\ell}(D)$, observe that $A_\ell,B_\ell$ are resolvents of the same real
symmetric matrix. Therefore, there is a real orthogonal matrix $O$ such that
\[
 A_\ell=O\diag(a_1,\ldots,a_M)O^T,\qquad
 B_\ell=O\diag(b_1,\ldots,b_M)O^T.
\]
where $a_j$ and $b_j$ are the eigenvalues of $A_\ell$ and $B_\ell$, respectively. Define the tensors $\widehat K^{(\gamma)}$ and
$\widehat C^{(\ell,\gamma)}$ in the orthonormal basis given by $O$:
\begin{align*}
 \widehat K^{(\gamma)}_{shjk}
 :=\sum_{a,b,c,d}
 (K_\gamma)_{abcd}O_{as}O_{bh}O_{cj}O_{dk},
 \qquad
 \widehat C^{(\ell,\gamma)}_{shjk}
 :=\sum_{a,b,c,d}
 C^{(\ell,\gamma)}_{abcd}O_{as}O_{bh}O_{cj}O_{dk}.
\end{align*}
And for fixed $h,j,k$, set
\[
 k^{(\gamma)}_{hjk}
 :=(\widehat K^{(\gamma)}_{shjk})_{s=1}^M,
 \qquad
 c^{(\ell,\gamma)}_{hjk}
 :=(\widehat C^{(\ell,\gamma)}_{shjk})_{s=1}^M.
\]
Then
\[
 \Theta_{\gamma,\ell}(D)
 =\sum_{h,j,k}a_hb_jb_k\,
 \bigl(k^{(\gamma)}_{hjk}\bigr)^T
 (O^TDO)c^{(\ell,\gamma)}_{hjk}.
\]
It follows that
\begin{equation}\label{eq.population-mixed-cumulant-bound}
 |\Theta_{\gamma,\ell}(D)|
 \le\frac{\|D\|_{\op}}2
 \sum_{s,h,j,k}|a_hb_jb_k|
 \left((\widehat K^{(\gamma)}_{shjk})^2
       +(\widehat C^{(\ell,\gamma)}_{shjk})^2\right).
\end{equation}
Both tensors on the right side are fully symmetric and separately satisfy the map bound
\eqref{eq.K-map-bound}, with constants uniform in $\gamma$ and $\ell$. Indeed, let $o_h$ denote the $h$th column of $O$. Using the full symmetry of $\widehat K^{(\gamma)}$
and $(o_h)_c=O_{ch}$, we obtain
\begin{align*}
 \widehat K^{(\gamma)}_{shjk}
 &=\frac12\left(
 \widehat K^{(\gamma)}_{skhj}
 +\widehat K^{(\gamma)}_{skjh}
 \right)\\
 &=\frac12\sum_{a,b,c,d}
 O_{as}(K_\gamma)_{abcd}
 \bigl(O_{ch}O_{dj}+O_{cj}O_{dh}\bigr)O_{bk}\\
 &=\left[
 O^T\mathcal K_\gamma\left(
 \frac{o_ho_j^T+o_jo_h^T}{2}
 \right)O
 \right]_{sk}.
\end{align*}
By \eqref{eq.K-map-bound}, we get
\[
 \sum_{s,k}(\widehat K^{(\gamma)}_{shjk})^2
 =\left\|
 O^T\cK_\gamma\left(
 \frac{o_ho_j^T+o_jo_h^T}{2}
 \right)O
 \right\|_{\HS}^2
 \le C\left\|\frac{o_ho_j^T+o_jo_h^T}{2}\right\|_{\HS}^2
 =C\frac{1+\delta_{hj}}2\le C,
\]
and the same inequality holds for $\widehat C^{(\ell,\gamma)}$.
Together with
\[
 |a_hb_jc_k|\le\frac13\left(
 |a_hb_j|^{3/2}+|a_hc_k|^{3/2}+|b_jc_k|^{3/2}\right),
\]
we conclude from \eqref{eq.population-mixed-cumulant-bound} that
\begin{align*}
 |\Theta_{\gamma,\ell}(D)|
 &\le C\sum_{s,h,j,k}|a_hb_jb_k|
 \left(
 (\widehat K^{(\gamma)}_{shjk})^2
 +(\widehat C^{(\ell,\gamma)}_{shjk})^2
 \right)\\
 &\le C\Biggl[
 \sum_{h,j}|a_hb_j|^{3/2}
 \sum_{s,k}\left(
 (\widehat K^{(\gamma)}_{shjk})^2
 +(\widehat C^{(\ell,\gamma)}_{shjk})^2
 \right)+
 \sum_{h,k}|a_hb_k|^{3/2}
 \sum_{s,j}\left(
 (\widehat K^{(\gamma)}_{shjk})^2
 +(\widehat C^{(\ell,\gamma)}_{shjk})^2
 \right)\\
 &\qquad+
 \sum_{j,k}|b_jb_k|^{3/2}
 \sum_{s,h}\left(
 (\widehat K^{(\gamma)}_{shjk})^2
 +(\widehat C^{(\ell,\gamma)}_{shjk})^2
 \right)
 \Biggr]\\
 &\le C\Biggl[
 \sum_{h,j}|a_h|^{3/2}|b_j|^{3/2}
 +\sum_{h,k}|a_h|^{3/2}|b_k|^{3/2}
 +\sum_{j,k}|b_j|^{3/2}|b_k|^{3/2}
 \Biggr]\\
 &=C\Biggl[
 2\left(\sum_h|a_h|^{3/2}\right)
  \left(\sum_j|b_j|^{3/2}\right)
 +\left(\sum_j|b_j|^{3/2}\right)^2
 \Biggr]\\
 &=C\left[
 2\Tr|A_\ell|^{3/2}\Tr|B_\ell|^{3/2}
 +\bigl(\Tr|B_\ell|^{3/2}\bigr)^2
 \right].
\end{align*}
The bounds on Lemma~\ref{lem:res} give
\[
 \bigl\|\Tr|A_\ell|^{3/2}\bigr\|_p
 +\bigl\|\Tr|B_\ell|^{3/2}\bigr\|_p
 \le C_pN^{1+C\epsilon},
\]
which imply $\|\Theta_{\gamma\ell}\|_p\le C_pN^{2+C\epsilon}$. Moreover, we have
\[
 \begin{aligned}
 |K_\gamma[D\Sigma A_\ell,B_\ell\Sigma B_\ell]|
 &\le C\|A_\ell\|_{\HS}\|B_\ell\|_{\op}\|B_\ell\|_{\HS},\\
 |K_\gamma[D\Sigma B_\ell,A_\ell\Sigma B_\ell]|
 &\le C\|A_\ell\|_{\op}\|B_\ell\|_{\HS}^2.
 \end{aligned}
\]
Their $L^p$ norms are at most $C_pN^{2+C\epsilon}$.
Substituting these estimates into \eqref{eq.EkVkall} proves
$\|\E_\ell V_\ell\|_p\le C_pN^{C\epsilon}$.
\end{proof}

\subsection{Proof of Lemma~\ref{lem.weighted-cumulant}}
\label{sec.weighted-cumulant-proof}
We now prove Lemma~\ref{lem.weighted-cumulant}
using Lemma~\ref{lem.internal-column}. The basic scale in Lemma~\ref{lem:res} gives
\[
\left|\mathbb E\bigl[W_\gamma K[R_1,R_2]\bigr]\right|
\le C N^{4/3+C\varepsilon}.
\]
This bound is insufficient for the Green function comparison argument. To obtain the required gain, we derive a self-consistent equation after deleting one additional column.
This yields a gain of $N^{1/6}$ over the basic scale bound. 

We only prove \eqref{eq.weighted-K11} here and the remaining estimates \eqref{eq.weighted-K21} and \eqref{eq.weighted-K22} follow from Cauchy's integral formula; The details are given in
Appendix \ref{app:cauchy-weighted-K}.

\begin{proof}[Proof of Lemma~\ref{lem.weighted-cumulant}]
Fix $\gamma$ and write $D=D(z_1)$. Multiplying
$z_1R_1=H_\gamma^{(\gamma)}R_1-I_M$ by $D$ gives
\begin{equation}\label{eq.K-resolvent-equation}
 z_1\E[W_\gamma K_\gamma[DR_1,R_2]]
 =\E[W_\gamma K_\gamma[DH_\gamma^{(\gamma)}R_1,R_2]]
  -\E[W_\gamma K_\gamma[D,R_2]].
\end{equation}
Since $\|D\|_{\HS}\le C\sqrt M$, we have
\[
 |\E[W_\gamma K_\gamma[D,R_2]]|
 \le C\sqrt M\,\|W_\gamma\|_2
                  \bigl\|\|R_2\|_{\HS}\bigr\|_2
 \le CN^{7/6+C\epsilon}.
\]
We next expand the first term on the right of \eqref{eq.K-resolvent-equation} after deleting an additional column $\ell\ne\gamma$. The Sherman Morrison formula yields
\[
 R_2=B_\ell-
       \frac{B_\ell v_\ell v_\ell^TB_\ell}
            {1+v_\ell^TB_\ell v_\ell},\qquad
 DH_\gamma^{(\gamma)}R_1
 =\sum_{\ell\ne\gamma}
       \frac{Dv_\ell v_\ell^TA_\ell}
            {1+v_\ell^TA_\ell v_\ell}.
\]
Consequently, we obtain
\begin{equation}\label{eq.K-column-expansion}
 \begin{aligned}
 z_1\E[W_\gamma K_\gamma[DR_1,R_2]]
 =\sum_{\ell\ne\gamma}
   \E\frac{W_\gamma U_\ell}{1+v_\ell^TA_\ell v_\ell}
   -\sum_{\ell\ne\gamma}
   \E\frac{W_\gamma V_\ell}
   {(1+v_\ell^TA_\ell v_\ell)(1+v_\ell^TB_\ell v_\ell)}
   +O(N^{7/6+C\epsilon}).
 \end{aligned}
\end{equation}

Lemma~\ref{lem.denominator-replacement} below allow us to replace the denominators used in \eqref{eq.K-column-expansion}. The proof of this lemma is given after the proof of Lemma~\ref{lem.weighted-cumulant}.
\begin{lemma}\label{lem.denominator-replacement}
Under the assumptions of Lemma~\ref{lem.weighted-cumulant},
\begin{equation}\label{eq.denominator-replacement}
 \begin{aligned}
 \sum_{\ell\ne\gamma}
   \E\frac{W_\gamma U_\ell}{1+v_\ell^TA_\ell v_\ell}
 &=d_0(z_1)^{-1}\sum_{\ell\ne\gamma}
                  \E[W_\gamma^{(\ell)}U_\ell]
   +O(N^{1+C\epsilon}),\\
 \sum_{\ell\ne\gamma}
   \E\frac{W_\gamma V_\ell}
   {(1+v_\ell^TA_\ell v_\ell)(1+v_\ell^TB_\ell v_\ell)}
 &=[d_0(z_1)d_0(z_2)]^{-1}\sum_{\ell\ne\gamma}
                  \E[W_\gamma^{(\ell)}V_\ell]
   +O(N^{7/6+C\epsilon}).
 \end{aligned}
\end{equation}
\end{lemma}

Combining \eqref{eq.K-column-expansion} and
\eqref{eq.denominator-replacement}, we obtain
\begin{equation}\label{eq.K-column-expansion2}
 z_1\E[W_\gamma K_\gamma[DR_1,R_2]]
 =\frac{\sum_{\ell\ne\gamma}\E[W_\gamma^{(\ell)}U_\ell]}{d_0(z_1)}
 -\frac{\sum_{\ell\ne\gamma}\E[W_\gamma^{(\ell)}V_\ell]}{d_0(z_1)d_0(z_2)}
 +O(N^{7/6+C\epsilon}).
\end{equation}

We first estimate the sum involving $V_\ell$.
Since $W_\gamma^{(\ell)}$ is fixed under $\E_\ell$,
the tower property and Lemma~\ref{lem.internal-column}
give
\begin{align*}
 \left|
 \sum_{\ell\ne\gamma}
 \E[W_\gamma^{(\ell)}V_\ell]
 \right|
 =\left|
 \sum_{\ell\ne\gamma}
 \E[W_\gamma^{(\ell)}\E_\ell V_\ell]
 \right|
 \le\sum_{\ell\ne\gamma}
 \|W_\gamma^{(\ell)}\|_2
 \|\E_\ell V_\ell\|_2
 \le CN^{1+C\epsilon}.
\end{align*}
Thus \eqref{eq.K-column-expansion} can simplify to
\begin{equation}\label{eq.K-quadratic-equation}
 z_1\E[W_\gamma K_\gamma[DR_1,R_2]]
 =\frac1{Nd_0(z_1)}\sum_{\ell\ne\gamma}
   \E[W_\gamma^{(\ell)}K_\gamma[D\Sigma A_\ell,B_\ell]]
   +O(N^{7/6+C\epsilon}).
\end{equation}
Here we use the Lemma~\ref{lem.internal-column} to yield
\[
 \E[W_\gamma^{(\ell)}U_\ell]
 =N^{-1}\E[W_\gamma^{(\ell)}K_\gamma[D\Sigma A_\ell,B_\ell]].
\]

To obtain a self-consistent equation in \eqref{eq.K-quadratic-equation} , we now replace $W_\gamma^{(\ell)}$ by $W_\gamma$ and $A_\ell,B_\ell$ by $R_1,R_2$, respectively. Indeed, the weight assumption and \eqref{eq.K-HS-bound} give
\[
 \E[(W_\gamma^{(\ell)}-W_\gamma)
                      K_\gamma[D\Sigma A_\ell,B_\ell]]
 =O(N^{1+C\epsilon}).
\]
Moreover, the Sherman Morrison formula and Lemma~\ref{lem:res} imply
\[
 \bigl\|\|A_\ell-R_1\|_{\HS}\bigr\|_p
 +\bigl\|\|B_\ell-R_2\|_{\HS}\bigr\|_p
 \le C_pN^{1/3+C\epsilon}.
\]
Together with \eqref{eq.K-HS-bound}, this yields
\[
 \E[W_\gamma K_\gamma[D\Sigma A_\ell,B_\ell]]
 =\E[W_\gamma K_\gamma[D\Sigma R_1,R_2]]+O(N^{1+C\epsilon}).
\]
For fixed $\gamma$, the tensor $K_\gamma$ is independent of the summation index $\ell$. The common covariance $\Sigma/N$ supplies the common leading coefficient. Substituting these estimates into \eqref{eq.K-quadratic-equation} and absorbing the missing summand gives
\begin{equation}\label{eq.K-self-consistent}
 z_1\E[W_\gamma K_\gamma[DR_1,R_2]]
 =d_0(z_1)^{-1}\E[W_\gamma K_\gamma[D\Sigma R_1,R_2]]
   +O(N^{7/6+C\epsilon}).
\end{equation}
By the definition of $D$, we have
$z_1D-d_0(z_1)^{-1}D\Sigma=I_M$. Therefore
\eqref{eq.K-self-consistent} reduces to
\[
 \E[W_\gamma K_\gamma[R_1,R_2]]=O(N^{7/6+C\epsilon}),
\]
which proves \eqref{eq.weighted-K11}.
\end{proof}

\begin{proof}[Proof of Lemma~\ref{lem.denominator-replacement}]
Setting
\[
 d_{\ell,1}:=1+v_\ell^TA_\ell v_\ell,\qquad
 d_{\ell,2}:=1+v_\ell^TB_\ell v_\ell,\qquad
 \widetilde d_a:=d_0(z_a),\qquad
 a=1,2.
\]
Since $v_\ell$ is independent of $A_\ell$ and $B_\ell$,
Lemma~\ref{lem:denominator} gives, for every fixed $p\ge1$,
\[
 \|d_{\ell,a}^{-1}\|_p+|\widetilde d_a^{-1}|\le C_p,
 \qquad
 \|d_{\ell,a}^{-1}-\widetilde d_a^{-1}\|_p
 \le C_pN^{-1/3+C\epsilon}.
\]
Using the identity
\[
 \frac1{d_{\ell,1}d_{\ell,2}}
 -\frac1{\widetilde d_1\widetilde d_2}
 =\left(\frac1{d_{\ell,1}}-\frac1{\widetilde d_1}\right)
    \frac1{d_{\ell,2}}
 +\frac1{\widetilde d_1}
    \left(\frac1{d_{\ell,2}}-\frac1{\widetilde d_2}\right)
\]
and H\"older's inequality, we yields
\[
 \left\|
 \frac{1}{d_{\ell,1}d_{\ell,2}}
 -\frac{1}{\widetilde d_1\widetilde d_2}
 \right\|_p
 \le C_pN^{-1/3+C\epsilon}.
\]

Combining these estimates with \eqref{eq.weight-assumptions}, we get
\[
 \begin{aligned}
 \left\|\frac{W_\gamma}{d_{\ell,1}}
           -\frac{W_\gamma^{(\ell)}}{\widetilde d_1}\right\|_2
 &= \left\|  \frac{W_\gamma-W_\gamma^{(\ell)}}{d_{\ell,1}}
 +W_\gamma^{(\ell)}
 \left(\frac{1}{d_{\ell,1}}-\frac{1}{\widetilde d_1}\right) \right\|_2
 \le CN^{-1/3+C\epsilon},\\
 \left\|\frac{W_\gamma}{d_{\ell,1}d_{\ell,2}}
           -\frac{W_\gamma^{(\ell)}}{\widetilde d_1\widetilde d_2}
       \right\|_2
 &=\left\|  =
 \frac{W_\gamma-W_\gamma^{(\ell)}}
      {d_{\ell,1}d_{\ell,2}} +
 W_\gamma^{(\ell)}
 \left(
 \frac{1}{d_{\ell,1}d_{\ell,2}}
 -\frac{1}{\widetilde d_1\widetilde d_2}
 \right) \right\|_2 \le CN^{-1/3+C\epsilon}.
 \end{aligned}
\]
Finally, multiply the first bound by
$\|U_\ell\|_2\le CN^{1/3+C\epsilon}$ and the second by
$\|V_\ell\|_2\le CN^{1/2+C\epsilon}$, and sum over
$\ell\ne\gamma$. This proves \eqref{eq.denominator-replacement}.
\end{proof}

\appendix

\section*{AI use}
GPT-5.6 and GPT-6 were used at various stages of this work, including technical derivations, grammar corrections and proofreading. The authors have independently reviewed the AI-assisted content and take full responsibility for the accuracy and integrity of the paper.

\section{Proof of Lemma~\ref{lem:res}}
\label{app:basic scales}
\begin{proof}[Proof of Lemma~\ref{lem:res}]
We first prove \eqref{eq:schatten}. Using the 
local law and spectral rigidity estimate in Lemma~\ref{lem:interpolation-stability}, and the square root behavior of the deformed MP density at the regular
right edge, we split the sum at $j_0=\lceil N^{3\varepsilon}\rceil$ and obtain,
for every fixed $s\ge1$, with arbitrarily high probability,
\begin{align*}
\operatorname{Tr}|R|^s
=\sum_{j=1}^M|\lambda_j-z|^{-s}
\le C_s\left(
N^{3\varepsilon}|\eta|^{-s}
+N^{2s/3}\sum_{j=j_0}^{\lfloor cN\rfloor}j^{-2s/3}
+N
\right).
\end{align*}
Here the first term follows from $|\lambda_j-z|\ge|\eta|$, the second
from rigidity and interlacing, and the last accounts for eigenvalues
bounded away from the upper edge, including any zero eigenvalues.
Evaluating the sum proves \eqref{eq:schatten}, for sufficiently small
$\varepsilon>0$. The Hilbert--Schmidt bounds \eqref{eq.HS-bounds}
follow from
\[
\|R\|_{\mathrm{HS}}^2=\operatorname{Tr}|R|^2,\qquad
\|R^2\|_{\mathrm{HS}}^2=\operatorname{Tr}|R|^4.
\]
For deterministic matrices $B_0,\ldots,B_k$ with bounded operator
norms, the Schatten H\"older inequality also gives
\[
\|B_0RB_1\cdots RB_k\|_r
\le C\|R\|_{kr}^{k},\qquad r\ge1,
\]
where the norms in this display are Schatten norms.
This proves the assertion concerning bounded deterministic insertions.
In particular, it applies to the factors of $\Sigma$ occurring below.

Conditioning on $R$, Assumption~\ref{assu.1.1} gives, for every complex
symmetric matrix $A$ determined by $R$,
\[
\bigl\|x^TAx-N^{-1}\operatorname{Tr}(\Sigma A)\bigr\|_{p}
\le C_pN^{-1}\|\Sigma^{1/2}A\Sigma^{1/2}\|_{\mathrm{HS}}
\le C_pN^{-1}\|A\|_{\mathrm{HS}}.
\]
Taking $A=R$ and using \eqref{eq.HS-bounds} yields
\[
x^TRx-N^{-1}\operatorname{Tr}(\Sigma R)
\prec N^{-1/3+C\varepsilon}.
\]
For $k=2,3$, taking $A=R^k$ gives
\[
|x^TR^kx|
\prec N^{-1}\left(
\operatorname{Tr}|R|^k+
(\operatorname{Tr}|R|^{2k})^{1/2}
\right)
\prec N^{2k/3-1+C\varepsilon}.
\]
We used $|\operatorname{Tr}(\Sigma R^k)|\le
\|\Sigma\|_{\op}\operatorname{Tr}|R|^k$.
Similarly, $R^*R$ is real symmetric, and hence
\[
|Rx|^2=x^TR^*Rx
\prec N^{-1}\left(
\operatorname{Tr}|R|^2+
(\operatorname{Tr}|R|^4)^{1/2}
\right)
\prec N^{1/3+C\varepsilon}.
\]

For $p\ge2$, the bound for $\|x\|_p$ follows from
$\|\Sigma\|_{\op}\le\tau^{-1}$ and the moment bounds for the
coordinates of $q$. Indeed,
\[
\|x\|_p^2
\le\frac{\|\Sigma\|_{\op}}N
\left\|\sum_{i=1}^M q_i^2\right\|_{p/2}
\le\frac C N\sum_{i=1}^M\|q_i\|_p^2
\le C_p\frac MN
\le C_p.
\]
The case $1\le p<2$ follows by monotonicity of $L^p$ norms.
Finally, $\|R\|_{\op}\le|\eta|^{-1}$ is bounded by a fixed power
of $N$ on \eqref{eq:domain1}. The same is true of all the Schatten
norms above. Splitting into the events of arbitrarily high probability
used above and their complements, and using the moment bounds for
$|x|$ on the complements, converts each stochastic estimate to its
stated $L^p$ bound, after increasing $C$.
This proves \eqref{eq:primitive}.
\end{proof}

\section{Proof of Lemma~\ref{lem:denominator}}
\label{app:denominator}
\begin{proof}[Proof of Lemma~\ref{lem:denominator}]
We first prove \eqref{eq:trace-mp}.
By Lemma~\ref{lem:interpolation-stability}, Theorem~\ref{local MP law}
also applies to $X_\gamma$. Let $u_a$ be an orthonormal eigenbasis of
$\Sigma$, with corresponding eigenvalues $\sigma_a$.
The anisotropic local law gives
\[
\begin{aligned}
\left|\frac1N\operatorname{Tr}
\bigl(\Sigma(G_{1,\gamma}(z)-\Pi(z))\bigr)\right|
\le\frac1N\sum_{a=1}^M\sigma_a
\left|u_a^T(G_{1,\gamma}(z)-\Pi(z))u_a\right|
\prec
\sqrt{\frac{|\operatorname{Im}m_2(z)|}{N|\eta|}}
+\frac1{N|\eta|}.
\end{aligned}
\]
The square root behavior at the regular edge and \eqref{eq:domain1}
imply that the right side is at most $N^{-1/3+C\varepsilon}$.
Furthermore, $R-G_{1,\gamma}(z)$ has rank at most $K_0$ and
operator norm at most $2|\eta|^{-1}$. Therefore,
\[
\frac1N\left|\operatorname{Tr}
\bigl(\Sigma(R-G_{1,\gamma}(z))\bigr)\right|
\le\frac{2K_0\|\Sigma\|_{\op}}{N|\eta|}.
\]
Combining the preceding estimates proves the first line of
\eqref{eq:trace-mp}. In particular, no averaged local law for the
weighted trace is needed here.

The self-consistent equation for $m_2$ gives
\[
1+\frac1N\operatorname{Tr}(\Sigma\Pi(z))
=1-\frac1{Nz}\operatorname{Tr}
\frac{\Sigma}{I_M+m_2(z)\Sigma}
=-\frac1{z m_2(z)}.
\]
At the regular right edge, $z$, $m_2(z)$ and their reciprocals are
bounded. The regularity Assumption~\ref{assu.population} on $\Sigma$ also gives
$\min_a|1+\sigma_am_2(z)|\ge c>0$ in the domain
\eqref{eq:domain1}. Thus $|d_0(z)|\asymp1$ and
\[
\left(zI_M-\frac\Sigma{d_0(z)}\right)^{-1}
=\bigl[z(I_M+m_2(z)\Sigma)\bigr]^{-1}
=-\Pi(z),\qquad \|\Pi(z)\|_{\op}\le C.
\]
This completes the proof of \eqref{eq:trace-mp}.

We next prove \eqref{eq:denominator-moments}.
The weighted trace estimate just proved and Lemma~\ref{lem:res} imply
\[
\begin{aligned}
\left|N^{-1}\operatorname{Tr}(\Sigma R)-(d_0(z)-1)\right|
&\prec N^{-1/3+C\varepsilon},\\
|x^TRx-(d_0(z)-1)|
&\le|x^TRx-N^{-1}\operatorname{Tr}(\Sigma R)|
+\left|N^{-1}\operatorname{Tr}(\Sigma R)-(d_0(z)-1)\right|\\
&\prec N^{-1/3+C\varepsilon}.
\end{aligned}
\]
Since $|d_0(z)|\ge c>0$, for sufficiently small $\varepsilon>0$,
\[
|1+N^{-1}\operatorname{Tr}(\Sigma R)|\ge c/2,\qquad
|1+x^TRx|\ge c/2
\]
with arbitrarily high probability.

For every realization, the positive semidefiniteness of
$H_\gamma^{(\mathbb T)}$ and $\Sigma$ gives
\[
\left|(1+N^{-1}\operatorname{Tr}(\Sigma R))^{-1}\right|
\le\frac{|z|}{|\operatorname{Im}z|},\qquad
\left|(1+x^TRx)^{-1}\right|
\le\frac{|z|}{|\operatorname{Im}z|}.
\]
Indeed, in an orthonormal eigenbasis $v_j$ of
$H_\gamma^{(\mathbb T)}$, both $N^{-1}\operatorname{Tr}(\Sigma R)$
and $x^TRx$ have the form $\sum_j a_j(\lambda_j-z)^{-1}$ with
$a_j\ge0$ and $\lambda_j\ge0$. Taking imaginary parts gives
\[
\left|\operatorname{Im}
\left[z\left(1+\sum_j\frac{a_j}{\lambda_j-z}\right)\right]\right|
=|\eta|\left(1+\sum_j\frac{a_j\lambda_j}{|\lambda_j-z|^2}\right)
\ge|\eta|.
\]
Since $|z|/|\operatorname{Im}z|\le CN^{2/3+\varepsilon}$,
the contribution of the exceptional event is negligible. Therefore,
for every fixed $p\ge1$,
\[
\left\|(1+N^{-1}\operatorname{Tr}(\Sigma R))^{-1}\right\|_p
+\left\|(1+x^TRx)^{-1}\right\|_p\le C_p.
\]
Finally,
\[
\left|\frac1{1+x^TRx}-\frac1{d_0(z)}\right|
=\frac{|x^TRx-(d_0(z)-1)|}{|1+x^TRx|\,|d_0(z)|}
\prec N^{-1/3+C\varepsilon}.
\]
Using the deterministic bound
\[
\left|\frac1{1+x^TRx}-\frac1{d_0(z)}\right|
\le\frac{|z|}{|\operatorname{Im}z|}+C,
\]
we obtain the claimed $L^p$ estimate after increasing $C$.
This completes the proof.
\end{proof}

\section{Proof of \texorpdfstring{\eqref{eq.weighted-K21} and \eqref{eq.weighted-K22}}{the differentiated cumulant estimates}}
\label{app:cauchy-weighted-K}
To derive \eqref{eq.weighted-K21} , we differentiate \eqref{eq.weighted-K11} with respect to the first spectral parameter, keeping $W_\gamma$ and $R_2$ fixed. For this purpose, let $\Gamma_1:=\{w\in\mathbb C:|w-z_1|=\rho\}$,
with $\rho=\tfrac14\min\{|\Im z_1|,|\Im z_2|\}$, and set
\[
f(w):=\E W_\gamma K_\gamma\bigl[
G_{1,\gamma}^{(\gamma)}(w),R_2
\bigr],
\]
It can be shown that the proof of \eqref{eq.weighted-K11} remains valid uniformly on $\Gamma_1$. Hence, we have
\[
\sup_{w\in\Gamma_1}|f(w)|
\le C N^{7/6+C\epsilon}.
\]
Since $\frac{\partial}{\partial w}G_{1,\gamma}^{(\gamma)}(w)
=\left(G_{1,\gamma}^{(\gamma)}(w)\right)^2$, combing with Cauchy's integral formula,  we have
\begin{equation*}\label{eq:app-cauchy-first}
\begin{aligned}
\left|\mathbb E W_\gamma K_\gamma[R_1^2,R_2]\right|
&=
\left|f'(z_1)\right|
=
\left|
\frac{1}{2\pi\mathrm i}
\oint_{\Gamma_1}
\frac{f(w)}{(w-z_1)^2}\,\mathrm dw \right|
\le
\frac{1}{2\pi}
\frac{2\pi\rho}{\rho^2}
\sup_{w\in\Gamma_1}|f(w)| \\
&=
\rho^{-1}\sup_{w\in\Gamma_1}|f(w)|
\le
C\eta^{-1}N^{7/6+C\varepsilon}
\le
CN^{11/6+C\varepsilon},
\end{aligned}
\end{equation*}
The last inequality uses $\eta^{-1}=N^{2/3+\varepsilon}$.
This proves \eqref{eq.weighted-K21}. Applying the same argument to the second
spectral parameter introduces one additional factor of
$\rho^{-1}=O(N^{2/3+\epsilon})$, which yields \eqref{eq.weighted-K22}.

% Requires \usepackage{hyperref} in the preamble.

\end{document}